\documentclass{article}%
\usepackage{mathrsfs}
\usepackage{amsmath}
\usepackage{amsfonts}
\usepackage{amssymb}
\usepackage{graphicx}%
\providecommand{\U}[1]{\protect \rule{.1in}{.1in}}
\newtheorem{theorem}{Theorem}[section]

\newtheorem{corollary}[theorem]{Corollary}

\newtheorem{definition}[theorem]{Definition}
\newtheorem{example}[theorem]{Example}

\newtheorem{lemma}[theorem]{Lemma}

\newtheorem{proposition}[theorem]{Proposition}
\newtheorem{remark}[theorem]{Remark}

\newenvironment{proof}[1][Proof]{\noindent \textbf{#1.} }{\  \rule{0.5em}{0.5em}}
\begin{document}

\title{A New Central Limit Theorem under Sublinear Expectations}
\author{Shige PENG\thanks{ The author
thanks the partial support from The National Basic Research Program
of China (973 Program) grant No. 2007CB814900 (Financial Risk).
}\\Institute of Mathematics\\Shandong University\\250100, Jinan,
China\\peng@sdu.edu.cn}
\date{version: March 18, 2008}
\maketitle

\begin{abstract}
We describe a new framework of a sublinear expectation space and the related
notions and results of distributions, independence. A new notion of
G-distributions is introduced which generalizes our G-normal-distribution in
the sense that mean-uncertainty can be also described. W present our new
result of central limit theorem under sublinear expectation. This theorem can
be also regarded as a generalization of the law of large number in the case of mean-uncertainty.

\end{abstract}

\section{Introduction}

The law of large numbers (LLN) and central limit theorem (CLT) are long and
widely been known as two fundamental results in the theory of probability and
statistics. A striking consequence of CLT is that accumulated independent and
identically distributed random variables tends to a normal distributed random
variable whatever is the original distribution. It is a very useful tool in
finance since many typical financial positions are accumulations of a large
number of small and independent risk positions. But CLT only holds in cases of
model certainty. In this paper we are interested in CLT with mean and variance-uncertainty.

Recently problems of model uncertainties in statistics, measures of risk and
superhedging in finance motivated us to introduce, in \cite{Peng2006a} and
\cite{Peng2006b} (see also \cite{Peng2004}, \cite{Peng2005} and references
herein), a new notion of sublinear expectation, called \textquotedblleft%
$G$-expectation\textquotedblright, and the related \textquotedblleft$G$-normal
distribution\textquotedblright \ (see Def. \ref{Def-Gnormal}) from which we
were able to define $G$-Brownian motion as well as the corresponding
stochastic calculus. The notion of $G$-normal distribution plays the same
important rule in the theory of sublinear expectation as that of normal
distribution in the classic probability theory. It is then natural and
interesting to ask if we have the corresponding LLN and CLT under a sublinear
expectation and, in particular, if the corresponding limit distribution of the
CLT is a $G$-normal distribution. This paper gives an affirmative answer. We
will prove that the accumulated risk positions converge `in law' to the
corresponding $G$-normal distribution, which is a distribution under sublinear
expectation. In a special case where the mean and variance uncertainty becomes
zero, the $G$-normal distribution becomes the classical normal distribution.
Technically we introduce a new method to prove a CLT under a sublinear
expectation space. This proof of our CLT is short since we borrow a deep
interior estimate of fully nonlinear PDE in \cite{Krylov}. \ The assumptions
of our CLT can be still improved.

This paper is organized as follows: in Section 2 we describe the framework of
a sublinear expectation space. The basic notions and results of distributions,
independence and the related product space of sublinear will be presented in
Section 3. In Section 4 we introduce a new notion of $G$-distributions which
generalizes our $G$-normal-distribution in the sense that mean-uncertainty can
be also described. Finally, in Section 5, we present our main result of CLT
under sublinear expectation. For reader's convenience we present some basic
results of viscosity solutions in the Appendix.

This paper is a new and generalized version of \cite{Peng2007a} in which only
variance uncertainty was considered for random variables instead random random
vectors. Our new CLT theorem can be applied to the case where both
mean-uncertainty and variance-uncertainty cannot be negligible. This theorem
can be also regarded as a new generalization of LLN. We refer to
\cite{Marinacci} and \cite{Marinacci1} for the developments of LLN with
non-additive probability measures.

\section{Basic settings}

For a given positive integer $n$ we will denote by $\left \langle
x,y\right \rangle $ the scalar product of $x$, $y\in \mathbb{R}^{n}$ and by
$\left \vert x\right \vert =(x,x)^{1/2}$ the Euclidean norm of $x$. We denote by
$\mathbb{S}(n)$ the collection of $n\times n$ symmetric matrices and by
$\mathbb{S}_{+}(n)$ the non negative elements in $\mathbb{S}(n)$. We observe
that $\mathbb{S}(n)$ is an Euclidean space with the scalar product
$\left \langle P,Q\right \rangle =tr[PQ]$.

In this paper we will consider the following type of spaces of sublinear
expectations: Let $\Omega$ be a given set and let $\mathcal{H}$ be a linear
space of real functions defined on $\Omega$ such that if $X_{1},\cdots
,X_{n}\in \mathcal{H}$ then $\varphi(X_{1},\cdots,X_{n})\in \mathcal{H}$ for
each $\varphi \in C_{l.Lip}(\mathbb{R}^{n})$ where $C_{l.Lip}(\mathbb{R}^{n})$
denotes the linear space of (local Lipschitz) functions $\varphi$ satisfying
\begin{align*}
|\varphi(x)-\varphi(y)|  &  \leq C(1+|x|^{m}+|y|^{m})|x-y|,\  \  \forall
x,y\in \mathbb{R}^{n}\text{, \ }\\
\  &  \text{for some }C>0\text{, }m\in \mathbb{N}\text{ depending on }\varphi.
\end{align*}
$\mathcal{H}$ is considered as a space of \textquotedblleft random
variables\textquotedblright. In this case we denote $X=(X_{1},\cdots,X_{n}%
)\in \mathcal{H}^{n}$.

\begin{remark}
In particular, if $X,Y\in \mathcal{H}$, then $|X|$, $X^{m}\in \mathcal{H}$ are
in $\mathcal{H}$. More generally $\varphi(X)\psi(Y)$ is still in $\mathcal{H}$
if $\varphi,\psi \in C_{l.Lip}(\mathbb{R})$.
\end{remark}

Here we use $C_{l.Lip}(\mathbb{R}^{n})$ in our framework only for some
convenience of technicalities. In fact our essential requirement is that
$\mathcal{H}$ contains all constants and, moreover, $X\in \mathcal{H}$ implies
$\left \vert X\right \vert \in \mathcal{H}$. In general $C_{l.Lip}(\mathbb{R}%
^{n})$ can be replaced by the following spaces of functions defined on
$\mathbb{R}^{n}$.

\begin{itemize}
\item { $\mathbb{L}^{\infty}(\mathbb{R}^{n})$: the space bounded
Borel-measurable functions; }

\item { $C_{b}(\mathbb{R}^{n})$: the space of bounded and continuous
functions; }

\item {$C_{b}^{k}(\mathbb{R}^{n})$: the space of bounded and $k$-time
continuously differentiable functions with bounded derivatives of all orders
less than or equal to $k$;}

\item { $C_{unif}(\mathbb{R}^{n})$: the space of bounded and uniformly
continuous functions; }

\item { $C_{b.Lip}(\mathbb{R}^{n})$: the space of bounded and Lipschitz
continuous functions.}
\end{itemize}

\begin{definition}
{\label{Def-1} { A \textbf{sublinear expectation }$\mathbb{\hat{E}}$ on
$\mathcal{H}$ is a functional $\mathbb{\hat{E}}:\mathcal{H}\mapsto \mathbb{R}$
satisfying the following properties: for all $X,Y\in \mathcal{H}$, we
have\newline \  \  \newline \textbf{(a) Monotonicity:}
\  \  \  \  \  \  \  \  \  \  \  \  \ If $X\geq Y$ then $\mathbb{\hat{E}}[X]\geq
\mathbb{\hat{E}}[Y].$\newline \textbf{(b) Constant preserving: \  \ }%
\  \ $\mathbb{\hat{E}}[c]=c$.\newline \textbf{(c)} \textbf{Sub-additivity:
\  \  \  \ }}}\  \  \  \  \  \  \  \ $\mathbb{\hat{E}}[X]-\mathbb{\hat{E}}%
[Y]\leq \mathbb{\hat{E}}[X-Y].$\newline{{\textbf{(d) Positive homogeneity: }
\ $\mathbb{\hat{E}}[\lambda X]=\lambda \mathbb{\hat{E}}[X]$,$\  \  \forall
\lambda \geq0$.\newline}}\newline(In many situation {{\textbf{(c) }}}is also
called property of self--domination). The triple $(\Omega,\mathcal{H}%
,\mathbb{\hat{E}}\mathbb{)}$ is called a \textbf{sublinear expectation space}
(compare with a probability space $(\Omega,\mathcal{F},\mathbb{P})$). If only
\textbf{(c) }and \textbf{(d) }are satisfied $\mathbb{\hat{E}}$ is called a
sublinear functional.
\end{definition}

\begin{remark}
Just as in the framework of a probability space, a sublinear expectation space
can be a completed Banach space under its natural norm $\left \Vert
\cdot \right \Vert =${{$\mathbb{\hat{E}}[\left \vert \cdot \right \vert ]$ (see
\cite{Peng2004}-\cite{Peng2007b}) and by using its natural capacity }}$\hat
{c}(\cdot)$ induced via {{$\mathbb{\hat{E}}[\left \vert \cdot \right \vert ]$
(see \cite{DenMa} and \cite{DHP}). But the results obtained in this paper do
not need the assumption of the space-completion. }}
\end{remark}

\begin{lemma}
\label{le0}\bigskip Let {{$\mathbb{E}$}} be a sublinear functional defined on
$(\Omega,\mathcal{H})$, i.e., \textbf{(c) }and \textbf{(d)} hold for
$\mathbb{E}$. Then there exists a family $\mathcal{Q}$ of linear functional on
$(\Omega,\mathcal{H})$ such that
\[
{{\mathbb{E}}}[X]:=\sup_{E\in \mathcal{Q}}E[X],\  \  \forall E\in \mathcal{Q}.
\]
and such that, for each $X\in \mathcal{H}$, there exists a $E\in \mathcal{Q}$
such that {{$\mathbb{E}$}}$[X]:=E[X]$. If we assume moreover that \textbf{(a)
}holds (resp.\textbf{ (a), (b)} hold) for $\mathbb{E}$, then \textbf{(a) }also
holds \ (resp.\textbf{ (a), (b)} hold) for each $E\in \mathcal{Q}$.
\end{lemma}

\begin{proof}
Let $\mathcal{Q}$ be the family of all linear functional dominated by
{{$\mathbb{E}$}}, i.e., $E[X]\leq \mathbb{E}[X]$, for all $X\in \mathcal{H}$,
$E\in \mathcal{Q}.$ We first prove that $\mathcal{Q}$ is non empty. For a given
$X\in \mathcal{H}$, we denote $L=\{aX:a\in \mathbb{R}\}$ which is a subspace of
$\mathcal{H}$. We define $I:L\rightarrow \mathbb{R}$ by $I[aX]=a{{\mathbb{E}}%
}[X]$, $\forall a\in \mathbb{R}$, then $I[\cdot]$ forms a linear functional on
$L$ and $I\leq${{$\mathbb{E}$}} on $L$. Since {{$\mathbb{E}$}}$[\cdot]$ is
sub-additive and positively homogeneous, by Hahn-Banach theorem (see e.g.
\cite{Yosida}pp102) there exists a linear functional $E$ on $\mathcal{H}$ such
that $E=I$ on $L$ and $E\leq{{\mathbb{E}}}$ on $\mathcal{H}$. Thus $E$ is a
linear functional dominated by ${{\mathbb{E}}}$ such that ${{\mathbb{E}}%
}[X]:=E[X]$. We now define%
\[
{{\mathbb{E}}}_{\mathcal{Q}}[X]\triangleq \sup_{E\in \mathcal{Q}}E[X].
\]
It is clear that $\mathbb{E}_{\mathcal{Q}}\mathbb{=}{{\mathbb{E}}}$.

If \textbf{(a) }holds for $\mathbb{E}$, then for each non negative element
$X\in$ $\mathcal{H}$, for each $E\in \mathcal{Q}$, $E[X]=-E[-X]\geq
-\mathbb{E}[-X]\geq0$, thus \textbf{(a)} also holds for $E$. If moreover
\textbf{(b) }holds for $\mathbb{E}$, then for each $c\in \mathbb{R}$,
$-E[c]=E[-c]\leq \mathbb{E}[-c]=-c$ and $E[c]\leq \mathbb{E}[c]=c$, we get
$E[c]=c$. The proof is complete.
\end{proof}

\begin{example}
{ For some $\varphi \in C_{l.Lip}(\mathbb{R})$, }$\xi \in \mathcal{H}$, let
{$\varphi(\xi)$ be a gain value favorable to a banker of a game. The banker
can choose among a set of distribution }${\{F(\theta,A)}\}_{A\in
\mathcal{B}(\mathbb{R}),\theta \in \Theta}${ of a random variable }$\xi${. In
this situation the robust expectation of the risk for a gamblers\ opposite to
the banker is:}
\[
\hat{\mathbb{E}}[\varphi(\xi)]:=\sup_{\theta \in \Theta}\int_{\mathbb{R}}%
\varphi(x)F(\theta,dx).
\]

\end{example}

\section{Distributions, independence and product spaces}

We now consider the notion of the distributions of random variables under
sublinear expectations. Let $X=(X_{1},\cdots,X_{n})$ be a given $n$%
-dimensional random vector on a sublinear expectation space $(\Omega
_{1},\mathcal{H}_{1},\mathbb{\hat{E}})$. We define a functional on
$C_{l.Lip}(\mathbb{R}^{n})$ by
\begin{equation}
\mathbb{\hat{F}}_{X}[\varphi]:=\mathbb{\hat{E}}[\varphi(X)]:\varphi \in
C_{l.Lip}(\mathbb{R}^{n})\mapsto(-\infty,\infty). \label{X-Distr}%
\end{equation}
The triple $(\mathbb{R}^{n},C_{l.Lip}(\mathbb{R}^{n}),\mathbb{\hat{F}}%
_{X}[\cdot])$ forms a sublinear expectation space. $\mathbb{\hat{F}}_{X}$ is
called the distribution of $X$.

\begin{definition}
Let $X_{1}$ and $X_{2}$ be two $n$--dimensional random vectors defined
respectively in {sublinear expectation spaces }$(\Omega_{1},\mathcal{H}%
_{1},\mathbb{\hat{E}}_{1})${ and }$(\Omega_{2},\mathcal{H}_{2},\mathbb{\hat
{E}}_{2})$. They are called identically distributed, denoted by $X_{1}%
\overset{d}{=}X_{2}$, if
\[
\mathbb{\hat{E}}_{1}[\varphi(X_{1})]=\mathbb{\hat{E}}_{2}[\varphi
(X_{2})],\  \  \  \forall \varphi \in C_{l.Lip}(\mathbb{R}^{n}).
\]
It is clear that $X_{1}\overset{d}{=}X_{2}$ if and only if their distributions coincide.
\end{definition}

\begin{remark}
If the distribution $\mathbb{\hat{F}}_{X}$ of $X\in \mathcal{H}$ is not a
linear expectation, then $X$ is said to have distributional uncertainty. The
distribution of $X$ has the following four typical parameters:
\[
\bar{\mu}:=\hat{\mathbb{E}}[X],\  \  \underline{\mu}:=-\mathbb{\hat{E}%
}[-X],\  \  \  \  \  \  \  \  \bar{\sigma}^{2}:=\hat{\mathbb{E}}[X^{2}],\  \  \underline
{\sigma}^{2}:=-\hat{\mathbb{E}}[-X^{2}].\  \
\]
The subsets $[\underline{\mu},\bar{\mu}]$ and $[\underline{\sigma}^{2}%
,\bar{\sigma}^{2}]$ characterize the mean-uncertainty and the
variance-uncertainty of $X$. The problem of zero-mean uncertainty have been
studied in [P3], [P4]. In this paper the mean uncertainty will be in our consideration.
\end{remark}

The following simple property is very useful in our sublinear analysis.

\begin{proposition}
{ \label{Prop-X+Y}Let $X,Y\in \mathcal{H}$ be such that $\mathbb{\hat{E}%
}[Y]=-\mathbb{\hat{E}}[-Y]$, i.e. }$Y$ has no mean uncertainty.{ Then we have%
\[
\mathbb{\hat{E}}[X+Y]=\mathbb{\hat{E}}[X]+\mathbb{\hat{E}}[Y].
\]
In particular, if $\mathbb{\hat{E}}[Y]=\mathbb{\hat{E}}[-Y]=0$, then
$\mathbb{\hat{E}}[X+Y]=\mathbb{\hat{E}}[X]$. }
\end{proposition}

\begin{proof}
{ It is simply because we have $\mathbb{\hat{E}}[X+Y]\leq \mathbb{\hat{E}%
}[X]+\mathbb{\hat{E}}[Y]$ and
\[
\mathbb{\hat{E}}[X+Y]\geq \mathbb{\hat{E}}[X]-\mathbb{\hat{E}}[-Y]=\mathbb{\hat
{E}}[X]+\mathbb{\hat{E}}[Y]\text{.}%
\]
}
\end{proof}

The following notion of independence plays a key role:

\begin{definition}
In a sublinear expectation space $(\Omega,\mathcal{H},\mathbb{\hat{E}})$ a
random vector $Y=(Y_{1},\cdots,Y_{n})$, $Y_{i}\in \mathcal{H}$ is said to be
independent to another random vector $X=(X_{1},\cdots,X_{m})$, $X_{i}%
\in \mathcal{H}$ under $\mathbb{\hat{E}}[\cdot]$ if for each test function
$\varphi \in C_{l.Lip}(\mathbb{R}^{m}\times \mathbb{R}^{n})$ we have
\[
\mathbb{\hat{E}}[\varphi(X,Y)]=\mathbb{\hat{E}}[\mathbb{\hat{E}}%
[\varphi(x,Y)]_{x=X}].
\]

\end{definition}

\begin{remark}
{In the case of linear expectation, this notion of independence is just the
classical one. It is important to note that under sublinear expectations the
condition \textquotedblleft$Y$ is independent to $X$\textquotedblright \ does
not implies automatically that \textquotedblleft$X$ is independent to
$Y$\textquotedblright. }
\end{remark}

\begin{example}
We consider a case where $X,Y\in \mathcal{H}$ are identically distributed and
$\mathbb{\hat{E}}[X]=\mathbb{\hat{E}}[-X]=0$ but $\bar{\sigma}^{2}%
=\mathbb{\hat{E}}[X^{2}]>\underline{\sigma}^{2}=-\mathbb{\hat{E}}[-X^{2}]$. We
also assume that $\mathbb{\hat{E}}[|X|]=\mathbb{\hat{E}}[X^{+}+X^{-}]>0$, thus
$\mathbb{\hat{E}}[X^{+}]=\frac{1}{2}\mathbb{\hat{E}}[|X|+X]=$$\frac{1}%
{2}\mathbb{\hat{E}}[|X|]>0$. In the case where $Y$ is independent to $X$, we
have%
\[
\mathbb{\hat{E}}[XY^{2}]=\mathbb{\hat{E}}[X^{+}\bar{\sigma}^{2}-X^{-}%
\underline{\sigma}^{2}]=(\bar{\sigma}^{2}-\underline{\sigma}^{2}%
)\mathbb{\hat{E}}[X^{+}]>0.
\]
But if $X$ is independent to $Y$ we have%
\[
\mathbb{\hat{E}}[XY^{2}]=0.
\]

\end{example}

The independence property of two random vectors $X,Y$ involves only the joint
distribution of $(X,Y)$. The following construction tells us how to construct
random vectors with given sublinear distributions and with joint independence.

\begin{definition}
\label{Def-2-10}Let $(\Omega_{i},\mathcal{H}_{i},$ $\mathbb{\hat{E}}_{i})$,
$i=1,2$ be two sublinear expectation spaces. We denote by
\begin{align*}
\mathcal{H}_{1}\times \mathcal{H}_{2}  &  :=\{Z(\omega_{1},\omega_{2}%
)=\varphi(X(\omega_{1}),Y(\omega_{2})):(\omega_{1},\omega_{2})\in \Omega
_{1}\times \Omega_{2},\  \\
\  \  &  \ (X,Y)\in(\mathcal{H}_{1})^{m}\times(\mathcal{H}_{2})^{n}%
,\  \varphi \in C_{l.Lip}(\mathbb{R}^{m}\times \mathbb{R}^{n}),\ m,n=1,2,\cdots
\},\  \
\end{align*}
and, for each random variable of the above form $Z(\omega_{1},\omega
_{2})=\varphi(X(\omega_{1}),Y(\omega_{2}))$,
\[
(\mathbb{\hat{E}}_{1}\times \mathbb{\hat{E}}_{2})\mathbb{[}Z]:=\mathbb{\hat{E}%
}_{1}\mathbb{[}\bar{\varphi}(X)],\  \  \text{where }\bar{\varphi}%
(x):=\mathbb{\hat{E}}_{2}\mathbb{[}\varphi(x,Y)],\ x\in \mathbb{R}^{m}.
\]
It is easy to check that the triple $(\Omega_{1}\times \Omega_{2}%
,\mathcal{H}_{1}\times \mathcal{H}_{2},\mathbb{\hat{E}}_{1}\times
\mathbb{\hat{E}}_{2}\mathbb{)}$ forms a sunlinear expectation space. We call
it the product space of sublinear expectation of $(\Omega_{1},\mathcal{H}%
_{1},$ $\mathbb{\hat{E}}_{1})$ and $(\Omega_{2},\mathcal{H}_{2},$
$\mathbb{\hat{E}}_{2})$. \ In this way we can define the product space of
sublinear expectation
\[
(%
%TCIMACRO{\dprod \limits_{i=1}^{n}}%
%BeginExpansion
{\displaystyle \prod \limits_{i=1}^{n}}
%EndExpansion
\Omega_{i},%
%TCIMACRO{\dprod \limits_{i=1}^{n}}%
%BeginExpansion
{\displaystyle \prod \limits_{i=1}^{n}}
%EndExpansion
\mathcal{H}_{i},%
%TCIMACRO{\dprod \limits_{i=1}^{n}}%
%BeginExpansion
{\displaystyle \prod \limits_{i=1}^{n}}
%EndExpansion
\mathbb{\hat{E}}_{i})
\]
of any given sublinear expectation spaces $(\Omega_{i},\mathcal{H}_{i},$
$\mathbb{\hat{E}}_{i})$, $i=1,2,\cdots,n$. In particular, when $(\Omega
_{i},\mathcal{H}_{i},$ $\mathbb{\hat{E}}_{i})=(\Omega_{1},\mathcal{H}_{1},$
$\mathbb{\hat{E}}_{1})$ we have the product space of the form $(\Omega
_{1}^{\otimes n},\mathcal{H}_{1}^{\otimes n},$ $\mathbb{\hat{E}}_{1}^{\otimes
n})$.
\end{definition}

The following property is easy to check.

\begin{proposition}
\label{Prop-2-10} \label{Prop12}Let $X_{i}$ be $n_{i}$-dimensional random
vectors in sublinear expectation spaces $(\Omega_{i},\mathcal{H}_{i},$
$\mathbb{\hat{E}}_{i})$, for $i=1,\cdots,n$, respectively. We denote
\[
Y_{i}(\omega_{1},\cdots,\omega_{n}):=X_{i}(\omega_{i}),\  \ i=1,\cdots,n.
\]
Then $Y_{i}$, $i=1,\cdots,n$ are random variables in the product space of
sublinear expectation $(%
%TCIMACRO{\dprod \limits_{i=1}^{n}}%
%BeginExpansion
{\displaystyle \prod \limits_{i=1}^{n}}
%EndExpansion
\Omega_{i},%
%TCIMACRO{\dprod \limits_{i=1}^{n}}%
%BeginExpansion
{\displaystyle \prod \limits_{i=1}^{n}}
%EndExpansion
\mathcal{H}_{i},%
%TCIMACRO{\dprod \limits_{i=1}^{n}}%
%BeginExpansion
{\displaystyle \prod \limits_{i=1}^{n}}
%EndExpansion
\mathbb{\hat{E}}_{i})$. Moreover we have $Y_{i}\overset{d}{=}X_{i}$ and
$Y_{i+1}$ is independent to $(Y_{1},\cdots,Y_{i})$, for each $i$.

Moreover, if $(\Omega_{i},\mathcal{H}_{i},\mathbb{\hat{E}}_{i})=(\Omega
_{1},\mathcal{H}_{1},\mathbb{\hat{E}}_{1})$ and $X_{i}=X_{1}$, for all $i$,
then we also have $Y_{i}\overset{d}{=}Y_{1}$. In this case $Y_{i}$ ia called
independent copies of $Y_{1}$ for $i=2,\cdots,n$.
\end{proposition}

{The situation \textquotedblleft$Y$ is independent to $X$\textquotedblright%
\ often appears when $Y$ occurs after $X$, thus a very robust sublinear
expectation should take the information of $X$ into account. We consider the
following example: Let }$Y=\psi(\xi,\theta)$, $\psi \in C_{b}(\mathbb{R}^{2})$,
where $\xi$ and $X$ are two bounded random variables in a classical
probability space $(\Omega,\mathcal{F},P)$ and $\theta$ is a completely
unknown parameter valued in a given interval $[a,b]$. We assume that $\xi$ is
independent of $X$ under $P$ in the classical sense. On the space
$(\Omega,\mathcal{H)}$ with $\mathcal{H}:=\{ \varphi(X,Y):\varphi \in
C_{l.Lip}(\mathbb{R}^{2})\}$, we can define the following three robust
sublinear expectations:
\begin{align*}
\mathbb{E}_{1}[\varphi(X,Y)] &  =\sup_{\theta \in \lbrack a,b]}E_{P}%
[\varphi(X,\psi(\xi,\theta)]\text{,\  \ }\mathbb{E}_{2}[\varphi(X,Y)]=E_{P}%
[\sup_{\theta \in \lbrack a,b]}\varphi(X,\psi(\xi,\theta)],\  \\
\mathbb{E}_{3}[\varphi(X,Y)] &  =E_{P}[\{ \sup_{\theta \in \lbrack a,b]}%
E_{P}[\varphi(x,\psi(\xi,\theta)]\}_{x=X}].\text{\ }%
\end{align*}
But it is seen that only under the sublinear expectation $\mathbb{E}_{3}$ that
$Y$ is independent to $X$.

\begin{remark}
It is possible that the above parameter $\theta$ is in fact a function of $X$
and $\xi$: $\theta=\Theta(X,\xi)$ where $\Theta$ is a completely unknown
function valued in $[a,b]$, thus $Y=\psi(\xi,\Theta(X,\xi))$ is dependent to
$X$ in the classical sense. But since $\Theta$ is a completely unknown
function a robust expectation is $\mathbb{E}_{3}$.
\end{remark}

\begin{definition}
A sequence of $d$-dimensional random vectors $\left \{  \eta_{i}\right \}
_{i=1}^{\infty}$ in $\mathcal{H}$ is said to converge in distribution under
$\mathbb{\hat{E}}$ if for each $\varphi \in C_{b}(\mathbb{R}^{n})$ the sequence
$\left \{  \mathbb{\hat{E}}[\varphi(\eta_{i})]\right \}  _{i=1}^{\infty}$ converges.
\end{definition}

\section{$G$-distributed random variables}

Given a pair of $d$-dimensional random vectors $(X,Y)$ in a sublinear
expectation space $(\Omega,\mathcal{H},\mathbb{\hat{E}})$, we can define a
function
\begin{equation}
G(p,A):=\mathbb{\hat{E}}[\frac{1}{2}\left \langle AX,X\right \rangle
+\left \langle p,Y\right \rangle ],\  \  \ (p,A)\in \mathbb{S}(d)\times
\mathbb{R}^{d} \label{G-func}%
\end{equation}
It is easy to check that $G:\mathbb{R}^{d}\times \mathbb{S}(d)\mapsto
\mathbb{R}$ is a sublinear function monotonic in $A\in \mathbb{S}(d)$ in the
following sense: For each $p,\bar{p}\in \mathbb{R}^{d}$ and $A,\bar{A}
\in \mathbb{S}(d)$%

\begin{equation}
\left \{
\begin{array}
[c]{rl}%
G(p+\bar{p},A+\bar{A}) & \leq G(p,A)+G(\bar{p},\bar{A}),\\
G(\lambda p,\lambda A) & =\lambda G(p,A),\  \  \forall \lambda \geq0,\\
G(p,A) & \geq G(p,\bar{A}),\ if\  \ A\geq \bar{A}.
\end{array}
\right.  \label{mono}%
\end{equation}
$G$ is also a continuous function.

The following property is classic. One can also check it by using Lemma
\ref{le0}.

\begin{proposition}
\label{Sublinear}Let $G:\mathbb{R}^{d}\times \mathbb{S}(d)\mapsto \mathbb{R}$ be
a sublinear function monotonic in $A\in \mathbb{S}(d)$ in the sense of
(\ref{mono}) and continuous in $(0,0)$. Then there exists a bounded subset
$\Theta \in \mathbb{R}^{d}\times \mathbb{R}^{d\times d}$ such that
\[
G(p,A)=\sup_{(q,Q)\in \Theta}[\frac{1}{2}tr[AQQ^{T}]+\left \langle
p,q\right \rangle ],\  \  \  \forall(p,A)\in \mathbb{R}^{d}\times \mathbb{S}(d).
\]

\end{proposition}

The classical normal distribution can be characterized through the notion of
stable distributions introduced by P. L\'{e}vy \cite{Levy1} and \cite{Levy2}.
\ The distribution of a $d$-dimensional random vector $X$ in a sublinear
expectation space $(\Omega,\mathcal{H},\mathbb{\hat{E}})$ is called stable if
for each $\mathbf{a}$, $\mathbf{b}\in \mathbb{R}^{d}$, there exists
$\mathbf{c\in}\mathbb{R}^{d}$ and $d\in \mathbb{R}$ such that
\[
\left \langle \mathbf{a},X\right \rangle +\left \langle \mathbf{b},\bar
{X}\right \rangle \overset{d}{=}\left \langle \mathbf{c},X\right \rangle +d,
\]
where $\bar{X}$ is an independent copy of $X$.

The following $G$-normal distribution plays the same role as normal
distributions in the classical probability theory:

\begin{proposition}
\label{Prop-Gnorm}Let $G:\mathbb{R}^{d}\times \mathbb{S}(d)\mapsto \mathbb{R}$
be a given sublinear function monotonic in $A\in \mathbb{S}(d)$ the sense of
(\ref{mono}) and continuous in $(0,0)$. Then there exists a pair of
$d$-dimensional random vectors $(X,Y)$ in some sublinear expectation space
$(\Omega,\mathcal{H},\mathbb{\hat{E}})$ satisfying (\ref{G-func}) and the
following condition:
\begin{equation}
(aX+b\bar{X},a^{2}Y+b^{2}\bar{Y})\overset{d}{=}(\sqrt{a^{2}+b^{2}}%
X,(a^{2}+b^{2})Y\ ),\  \  \  \forall a,b\geq0,\label{G-normal}%
\end{equation}
where $(\bar{X},\bar{Y})$ is an independent copy of $(X,Y)$. The distribution
of $(X,Y)$ is uniquely determine by $G$.
\end{proposition}

\begin{example}
\label{Prop-G1}For the sublinear function $\bar{G}:\mathbb{R}^{d}%
\mapsto \mathbb{R}$ defined by $\bar{G}(p):=G(p,0)$, $p\in \mathbb{R}^{d}$, we
can concretely construct a $d$-dimensional random vector $Y$ in some sublinear
expectation space $(\Omega,\mathcal{H},\mathbb{\hat{E}})$ satisfying
\begin{equation}
\bar{G}(p):=\mathbb{\hat{E}}[\left \langle p,Y\right \rangle ],\  \  \ p\in
\mathbb{R}^{d}\label{plinear}%
\end{equation}
and the following condition:
\begin{equation}
a^{2}Y+b^{2}\bar{Y}\overset{d}{=}(a^{2}+b^{2})Y,\  \  \  \forall a,b\in
\mathbb{R},\  \label{a2b2}%
\end{equation}
where $Y$ is an independent copy of $Y$. In fact we can take $\Omega
=\mathbb{R}^{d}$, $\mathcal{H}=C_{l.Lip}(\mathbb{R}^{d})$ and $Y(\omega
)=\omega$. To define the corresponding sublinear expectation $\mathbb{\hat{E}%
}$, we apply Proposition \ref{Sublinear} to find a subset $\bar{\Theta}%
\in \mathbb{R}^{d}$ such that
\begin{equation}
\bar{G}(p)=\sup_{q\in \bar{\Theta}}\left \langle p,q\right \rangle ,\  \  \  \ p\in
\mathbb{R}^{d}.\label{thetabar}%
\end{equation}
Then for each $\xi \in \mathcal{H}$ of the form $\xi(\omega)=\varphi(\omega)$,
$\varphi \in C_{l.Lip}(\mathbb{R}^{d})$. $\omega \in \mathbb{R}^{d}$ we set
\begin{equation}
\mathbb{\hat{E}}[\xi]=\sup_{\omega \in \bar{\Theta}}\varphi(\omega
).\label{supOmega}%
\end{equation}
It is easy to check that the distribution of $Y$ satisfies (\ref{plinear}) and
(\ref{a2b2}). It is the so-called worst case distribution with respect to the
subset of mean uncertainty $\bar{\Theta}$. We denote this distribution by
$\mathcal{U}(\bar{\Theta})$.
\end{example}

\begin{example}
\label{Prop-G0}For the sublinear and monotone function $\hat{G}:\mathbb{S}%
(d)\mapsto \mathbb{R}$ defined by $\hat{G}(A):=G(0,A)$, $A\in \mathbb{S}(d)$ the
$d$-dimensional random vector $X$ in Proposition \ref{Prop-Gnorm} satisfies
\begin{equation}
\hat{G}(A):=\frac{1}{2}\mathbb{\hat{E}}[\left \langle AX,X\right \rangle
],\  \  \ p\in \mathbb{R}^{d}\label{Qlinear}%
\end{equation}
and the following condition:
\begin{equation}
aX+b\bar{X}\overset{d}{=}\sqrt{a^{2}+b^{2}}X,\  \  \  \forall a,b\in
\mathbb{R},\  \label{srt-a2b2}%
\end{equation}
where $\bar{X}$ is an independent copy of $X$. In particular, for each
components $X_{i}$ of $X$ and $\bar{X}_{i}$ of $\bar{X}$, we have $\sqrt
{2}\mathbb{\hat{E}}[X_{i}]=\mathbb{\hat{E}}[X_{i}+\bar{X}_{i}]=2\mathbb{\hat
{E}}[X_{i}]$ and $\sqrt{2}\mathbb{\hat{E}}[-X_{i}]=\mathbb{\hat{E}}%
[-X_{i}-\bar{X}_{i}]=2\mathbb{\hat{E}}[-X_{i}]$ it follows that $X$ has no
mean uncertainty:
\[
\mathbb{\hat{E}}[X_{i}]=\mathbb{\hat{E}}[-X_{i}]=0,\  \ i=1,\cdots,d.
\]
On the other hand, by Proposition \ref{Sublinear} we can find a bounded subset
$\hat{\Theta}\in \mathbb{S}_{+}(d)$ such that
\begin{equation}
\frac{1}{2}\mathbb{\hat{E}}[\left \langle AX,X\right \rangle ]=\hat{G}%
(A)=\frac{1}{2}\sup_{Q\in \bar{\Theta}}tr[AQ],\  \ A\in \mathbb{S}%
(d).\label{Theta}%
\end{equation}
If $\hat{\Theta}$ is a singleton $\hat{\Theta}=\{Q\}$, then $X$ is a classical
zero-mean normal distributed with covariance $Q$. In general $\hat{\Theta}$
characterizes the covariance uncertainty of $X$.
\end{example}

\begin{definition}
(\textbf{$G$-distribution}) \label{Def-Gnormal} The pair of $d$-dimensional
random vectors $(X,Y)$ in the above proposition is called $G$-distributed. $X$
is said to be $\hat{G}$-normal distributed. We denote the distribution of $X$
by $X\overset{d}{=}\mathcal{N}(0,\hat{\Theta})$.
\end{definition}

Proposition \ref{gG-P1} and Corollary \ref{gG-P1coro} show that a
$G$-distribution is a uniquely defined sublinear distribution on
$(\mathbb{R}^{2d},C_{l.Lip}(\mathbb{R}^{2d}))$. We will show that a pair of
$G$-distributed random vectors{ is characterized, or generated, by the
following parabolic PDE }defined on $[0,\infty)\times \mathbb{R}^{d}%
\times \mathbb{R}^{d}$:
\begin{equation}
\partial_{t}u-G(D_{y}u,D_{x}^{2}u)=0, \label{eq-G-heat}%
\end{equation}
with Cauchy condition$\  \ u|_{t=0}=\varphi$, where $D_{y}=(\partial_{y_{i}%
})_{i=1}^{d}$, $D_{x}^{2}=(\partial_{x_{i},x_{j}}^{2})_{i,j=1}^{d}$.
(\ref{eq-G-heat}) is called the $G$-heat equation.

\begin{remark}
We will use the notion of viscosity solutions to the generating heat equation
(\ref{eq-G-heat}). This notion was introduced by Crandall and Lions. For the
existence and uniqueness of solutions and related very rich references we
refer to Crandall, Ishii and Lions \cite{CIL} (see Appendix for the
uniqueness). We note that, in the situation where $\underline{\sigma}^{2}>0$,
the viscosity solution (\ref{eq-G-heat}) becomes a classical $C^{1+\frac
{\alpha}{2},2+\alpha}$-solution (see \cite{Krylov} and the recent works of
\cite{Caff1997} and \cite{WangL}). Readers can understand (\ref{eq-G-heat}) in
the classical meaning.
\end{remark}

\begin{definition}
\text{ A} real-valued continuous function $u\in C([0,T]\times \mathbb{R}^{d})$
is called a viscosity subsolution (respectively, supersolution) of
(\ref{eq-G-heat}) if, for each function $\psi \in C_{b}^{3}((0,\infty
)\times \mathbb{R}^{d}\times \mathbb{R}^{d})$ and for each minimum
(respectively, maximum) point $(t,x,y)\in(0,\infty)\times \mathbb{R}^{d}%
\times \mathbb{R}^{d}$ of $\psi-u$, we have
\[
\partial_{t}\psi-G(D_{y}\psi,D_{x}^{2}\psi)\leq0\  \ (\text{respectively, }%
\geq0).
\]
$u$ is called a viscosity solution of (\ref{eq-G-heat}) if it is both super
and subsolution.
\end{definition}

\begin{proposition}
\label{gG-P1}Let $(X,Y)$ be $G$-distributed. For each $\varphi \in
C_{l.Lip}(\mathbb{R}^{d}\times \mathbb{R}^{d})$ we define a function%
\[
u(t,x,y):=\mathbb{\hat{E}}[\varphi((x+\sqrt{t}X,y+tY)],\ (t,x)\in
\lbrack0,\infty)\times \mathbb{R}.\
\]
Then we have%
\begin{equation}
u(t+s,x,y)=\mathbb{\hat{E}}[u(t,x+\sqrt{s}X,y+sY)],\  \ s\geq0. \label{DPP}%
\end{equation}
We also have the estimates: For each $T>0$ there exist constants $C,k>0$ such
that, for all $t,s\in \lbrack0,T]$ and $x,y\in \mathbb{R}$,
\begin{equation}
|u(t,x,y)-u(t,\bar{x},\bar{y})|\leq C(1+|x|^{k}+|y|^{k}+|\bar{x}|^{k}+|\bar
{y}|^{k})|x-y| \label{|x-y|}%
\end{equation}
and%
\begin{equation}
|u(t,x,y)-u(t+s,x+y)|\leq C(1+|x|^{k}+|y|^{k})(s+|s|^{1/2}). \label{|s|}%
\end{equation}
Moreover, $u$ is the unique viscosity solution, continuous in the sense of
(\ref{|x-y|}) and (\ref{|s|}), of the generating PDE (\ref{eq-G-heat}).
\end{proposition}

\begin{proof}
Since%
\begin{align*}
u(t,x,y)-u(t,\bar{x},\bar{y}) &  =\mathbb{\hat{E}}[\varphi(x+\sqrt
{t}X,y+tY)]-\mathbb{\hat{E}}[\varphi(\bar{x}+\sqrt{t}X,\bar{y}+tY)]\\
&  \leq \mathbb{\hat{E}}[\varphi(x+\sqrt{t}X,y+tY)-\varphi(\bar{x}+\sqrt
{t}X,\bar{y}+tY)]\\
&  \leq \mathbb{\hat{E}}[C_{1}(1+|X|^{k}+|Y|^{k}+|x|^{k}+|y|^{k}+|\bar{x}%
|^{k}+|\bar{y}|^{k})]\\
&  \  \  \  \  \times(|x-\bar{x}|+|y-\bar{y}|)\\
&  \leq C(1+|x|^{k}+|y|^{k}+|\bar{x}|^{k}+|\bar{y}|^{k})(|x-\bar{x}%
|+|y-\bar{y}|).
\end{align*}
We then have (\ref{|x-y|}). Let $(\bar{X},\bar{Y})$ be an independent copy of
$(X,Y)$. Since $(X,Y)$ is $G${-distributed, then}%
\begin{align*}
u(t+s,x,y) &  =\mathbb{\hat{E}}[\varphi(x+\sqrt{t+s}X,y+(t+s)Y)]\\
&  =\mathbb{\hat{E}}[\varphi(x+\sqrt{s}X+\sqrt{t}\bar{X},y+sY+t\bar{Y})]\\
&  =\mathbb{\hat{E}}[\mathbb{\hat{E}}[\varphi(x+\sqrt{s}\widetilde{x}+\sqrt
{t}\bar{X},y+s\widetilde{y}+t\bar{Y})]_{(\widetilde{x},\widetilde{y}%
)=(X,Y)}]\\
&  =\mathbb{\hat{E}}[u(t,x+\sqrt{s}X,y+sY)].
\end{align*}
We thus obtain (\ref{DPP}). From this and (\ref{|x-y|}) it follows that
\begin{align*}
&  u(t+s,x,y)-u(t,x,y)=\mathbb{\hat{E}}[u(t,x+\sqrt{s}X,y+sY)-u(t,x)]\\
&  \leq \mathbb{\hat{E}}[C_{1}(1+|x|^{k}+|y|^{k}+|X|^{k}+|Y|^{k})(\sqrt
{s}|X|+s|Y|)].
\end{align*}
Thus we obtain (\ref{|s|}). Now, for a fixed $(t,x,y)\in(0,\infty
)\times \mathbb{R}^{d}\times \mathbb{R}^{d}$, let $\psi \in C_{b}^{1,3}%
([0,\infty)\times \mathbb{R}^{d}\times \mathbb{R}^{d})$ be such that $\psi \geq
u$ and $\psi(t,x,y)=u(t,x,y)$. By (\ref{DPP}) and Taylor's expansion it
follows that, for $\delta \in(0,t)$,
\begin{align*}
0 &  \leq \mathbb{\hat{E}}[\psi(t-\delta,x+\sqrt{\delta}X,y+\delta
Y)-\psi(t,x,y)]\\
&  \leq \bar{C}(\delta^{3/2}+\delta^{2})-\partial_{t}\psi(t,x,y)\delta \\
&  +\mathbb{\hat{E}}[\left \langle D_{x}\psi(t,x,y),X\right \rangle \sqrt
{\delta}+\left \langle D_{y}\psi(t,x,y),Y\right \rangle \delta+\frac{1}%
{2}\left \langle D_{x}^{2}\psi(t,x,y)X,X\right \rangle \delta]\\
&  =-\partial_{t}\psi(t,x,y)\delta+\mathbb{\hat{E}}[\left \langle D_{y}%
\psi(t,x,y),Y\right \rangle +\frac{1}{2}\left \langle D_{x}^{2}\psi
(t,x,y)X,X\right \rangle ]\delta+\bar{C}(\delta^{3/2}+\delta^{2})\\
&  =-\partial_{t}\psi(t,x,y)\delta+\delta G(D_{y}\psi,D_{x}^{2}\psi
)(t,x,y)+\bar{C}(\delta^{3/2}+\delta^{2}).
\end{align*}
From which it is easy to check that
\[
\lbrack \partial_{t}\psi-G(D_{y}\psi,D_{x}^{2}\psi)](t,x,y)\leq0.
\]
Thus $u$ is a viscosity supersolution of (\ref{eq-G-heat}). Similarly we can
prove that $u$ is a viscosity subsolution of (\ref{eq-G-heat}).
\end{proof}

\begin{corollary}
\label{gG-P1coro} If both $(X,Y)$ and $(\bar{X},\bar{Y})$ are $G$-{distributed
with the same }$G${, i.e.,}%
\[
G(p,A):=\mathbb{\hat{E}}[\frac{1}{2}\left \langle AX,X\right \rangle
+\left \langle p,Y\right \rangle ]=\mathbb{\hat{E}}[\frac{1}{2}\left \langle
A\bar{X},\bar{X}\right \rangle +\left \langle p,\bar{Y}\right \rangle
],\  \  \forall(p,A)\in \mathbb{S}(d)\times \mathbb{R}^{d}.
\]
{ then }$(X,Y)\overset{d}{=}(\bar{X},\bar{Y})$. In particular, $X\overset
{d}{=}-X$.
\end{corollary}

\begin{proof}
For each $\varphi \in C_{l.Lip}(\mathbb{R}^{d}\times \mathbb{R}^{d})$ we set
\begin{align*}
u(t,x,y) &  :=\mathbb{\hat{E}}[\varphi(x+\sqrt{t}X,y+tY)],\  \\
\bar{u}(t,x,y) &  :=\mathbb{\hat{E}}[\varphi(x+\sqrt{t}\bar{X},y+t\bar
{Y})],\ (t,x)\in \lbrack0,\infty)\times \mathbb{R}.
\end{align*}
By the above Proposition, both $u$ and $\bar{u}$ are viscosity solutions of
the $G$-heat equation (\ref{eq-G-heat}) with Cauchy condition $u|_{t=0}%
=\bar{u}|_{t=0}=\varphi$. It follows from the uniqueness of the viscosity
solution that $u\equiv \bar{u}$. In particular%
\[
\mathbb{\hat{E}}[\varphi(X,Y)]=\mathbb{\hat{E}}[\varphi(\bar{X},\bar
{Y})]\text{.}%
\]
Thus $(X,Y)\overset{d}{=}(\bar{X},\bar{Y})$.
\end{proof}

\begin{corollary}
Let $(X,Y)$ be $G$-distributed. For each $\psi \in C_{l.Lip}(\mathbb{R}^{d})$
we define a function%
\[
v(t,x):=\mathbb{\hat{E}}[\psi((x+\sqrt{t}X+tY)],\ (t,x)\in \lbrack
0,\infty)\times \mathbb{R}^{d}.\
\]
Then $v$ is the unique viscosity solution of the following parabolic PDE%
\begin{equation}
\partial_{t}v-G(D_{x}v,D_{x}^{2}v)=0,\  \  \  \ v|_{t=0}=\psi. \label{gG-heat}%
\end{equation}
Moreover we have $v(t,x+y)\equiv u(t,x,y)$, where $u$ is the solution of the
PDE (\ref{eq-G-heat}) with initial condition $u(t,x,y)|_{t=0}=\psi(x+y)$.
\end{corollary}

\subsection{Proof of Proposition \ref{Prop-Gnorm}}

We now proceed to prove Proposition \ref{Prop-Gnorm}. Let $u=u^{\varphi}$ be
the unique viscosity solution of the $G$-heat equation (\ref{eq-G-heat}) with
$u^{\varphi}|_{t=0}=\varphi$, Then we take $\widetilde{\Omega}=\mathbb{R}%
^{2d}$, $\mathcal{\widetilde{H}}=C_{l.Lip}(\mathbb{R}^{2d})$, $\widetilde
{\omega}=(x,y)\in \mathbb{R}^{2d}$. The corresponding sublinear expectation
$\mathbb{\widetilde{E}}[\cdot]$ is defined by, for each $\xi \in \mathcal{H}$ of
the form $\xi(\omega)=(\varphi(x,y))_{(x,y)\in \mathbb{R}^{2d}}\in
C_{l.Lip}(\mathbb{R}^{2d})$, $\mathbb{\widetilde{E}}[\xi]=u^{\varphi}(1,0)$.
The monotonicity and sub-linearity of $u^{\varphi}$ with respect to $\varphi$
are known in the theory of viscosity solution. For reader's convenience we
provide a new and simple proof in the Appendix (see Corollary \ref{Comparison}
and Corollary \ref{Domination}). The positive homogeneity of
$\mathbb{\widetilde{E}}[\cdot]$ is easy to be checked.

We now consider a pairs of $d$-dimensional random vectors $(\widetilde
{X},\widetilde{Y})(\omega)=(x,y)$. We have
\[
\mathbb{\hat{E}}[\varphi(\widetilde{X},\widetilde{Y})]=u^{\varphi
}(1,0),\  \  \  \  \forall \varphi \in C_{l.Lip}(\mathbb{R}^{2d}).
\]
In particular, just set $\varphi_{0}(x,y)=\frac{1}{2}\left \langle
Ax,x\right \rangle +\left \langle p,y\right \rangle $, we can check that
\[
u^{\varphi_{0}}(t,x,y):=G(p,A)t+\frac{1}{2}\left \langle Ax,x\right \rangle
+\left \langle p,y\right \rangle .
\]
We thus have
\[
\mathbb{\widetilde{E}}[\frac{1}{2}\left \langle A\widetilde{X},\widetilde
{X}\right \rangle +\left \langle p,\widetilde{Y}\right \rangle ]=u^{\varphi_{0}%
}(t,0)|_{t=1}=G(p,A),\  \ (p,A)\in \mathbb{R}^{d}\times \mathbb{S}(n).
\]

To prove that the distribution of $(\widetilde{X},\widetilde{Y})$ satisfies
condition (\ref{G-normal}), we follow Definition \ref{Def-2-10} to construct a
product space of sublinear expectation
\[
(\Omega,\mathcal{H},\mathbb{\hat{E}})=(\widetilde{\Omega}\times \widetilde
{\Omega},\mathcal{\widetilde{H}\times \widetilde{H}},\mathbb{\widetilde
{E}\times \widetilde{E}})
\]
and introduce two pair of random vectors%
\[
(X,Y)(\omega_{1},\omega_{2})=\omega_{1},\ (\bar{X},\bar{Y})(\omega_{1}%
,\omega_{2})=\omega_{2},\  \  \ (\omega_{1},\omega_{2})\in \widetilde{\Omega
}\times \widetilde{\Omega}.\
\]
By Proposition \ref{Prop12} both $(X,Y)\overset{d}{=}(\widetilde{X}%
,\widetilde{Y})$ and $(\bar{X},\bar{Y})$ is an independent copy of $(X,Y)$.
For each $\varphi \in C_{l.Lip}(\mathbb{R}^{2d})$ and for each fixed
$\lambda>0$, $(\bar{x},\bar{y})\in \mathbb{R}^{2d}$, since the function $v$
defined by $v(t,x,y):=u^{\varphi}(\lambda t,\bar{x}+\sqrt{\lambda}x,\bar
{y}+\lambda y)$ solves exactly the same equation (\ref{eq-G-heat}) but with
Cauchy condition
\[
v|_{t=0}=\varphi(\bar{x}+\sqrt{\lambda}\times \cdot,\bar{y}+\lambda \times
\cdot).
\]
Thus
\begin{align*}
\mathbb{\hat{E}}[\varphi(\bar{x}+\sqrt{\lambda}X,\bar{y}+\lambda Y)] &
=v(t,\bar{x},\bar{y})|_{t=1}\\
&  =u^{\varphi(\sqrt{\lambda}\times \cdot,\lambda \times \cdot)}(t,\bar{x}%
,\bar{y})|_{t=1}=u^{\varphi}(\lambda,\bar{x},\bar{y}).\
\end{align*}
By the definition of $\mathbb{\hat{E}}$, for each $t>0$ and $s>0$,
\begin{align*}
\mathbb{\hat{E}}[\varphi(\sqrt{t}X+\sqrt{s}\bar{X},tY+s\bar{Y})] &
=\mathbb{\widetilde{E}}[\mathbb{\widetilde{E}}[\varphi(\sqrt{t}x+\sqrt{s}%
\bar{X},ty+s\bar{Y})]_{(x,y)=(X,Y)}]\\
&  =u^{u^{\varphi}(s,\cdot,\cdot)}(t,0,0)=u^{\varphi}(t+s,0,0)\\
&  =\mathbb{\hat{E}}[\varphi(\sqrt{t+s}X,(t+s)Y)].
\end{align*}
Namely $(\sqrt{t}X+\sqrt{s}\bar{X},tY+s\bar{Y})\overset{d}{=}(\sqrt
{t+s}X,(t+s)Y)$. Thus the distribution of $(X,Y)$ satisfies condition
(\ref{G-normal}){.}

It remains to check that the functional $\mathbb{\widetilde{E}}[\cdot
]:C_{l.Lip}(\mathbb{R}^{2d})\mapsto \mathbb{R}$ forms a sublinear expectation,
i.e., (a)-(d) of Definition \ref{Def-1} are satisfied. Indeed, (a) is simply
the consequence of comparison theorem, or the maximum principle of viscosity
solution (see [CIL], the prove of this comparison theorem as well as the
sub-additivity (c) are given in the Appendix of [P6]). It is also easy to
check that, when $\varphi \equiv c$, the unique solution of (\ref{eq-G-heat})
is also $u\equiv c$; hence (b) holds true. (d) also holds since $u^{\lambda
\varphi}=\lambda u^{\varphi}$, $\lambda \geq0$. \ The proof is complete.

\section{Central Limit Theorem{ }}

\begin{theorem}
\label{CLT}{(Central Limit Theorem) Let a sequence $\left \{  (X_{i}%
,Y_{i})\right \}  _{i=1}^{\infty}$ of }$\mathbb{R}^{d}\times \mathbb{R}^{d}%
$-valued random variables in{ $(\mathcal{H},\mathbb{\hat{E}})$}. We assumed
that $(X_{i+1},Y_{i+1})\overset{d}{=}(X_{i},Y_{i})${ and $(X_{i+1},Y_{i+1})$
is independent to $\{(X_{1},Y_{1}),\cdots,(X_{i},Y_{i})\}$ for each
$i=1,2,\cdots$. We assume furthermore that,
\[
\mathbb{\hat{E}}[X_{1}]=\mathbb{\hat{E}}[-X_{1}]=0,\  \
\]
Then the sequence }$\{ \bar{S}_{n}\}_{n=1}^{\infty}$ defined by{ }%
\[
\bar{S}_{n}:=\sum_{i=1}^{n}(\frac{X_{i}}{\sqrt{n}}+\frac{Y_{i}}{n})
\]
{converges in law to }$\xi+\zeta${:
\begin{equation}
\lim_{n\rightarrow \infty}\mathbb{\hat{E}}[\varphi(\bar{S}_{n}%
)]=\mathbb{\widetilde{E}}[\varphi(\xi+\zeta)],\  \  \  \  \label{e12}%
\end{equation}
}for all functions $\varphi \in C(\mathbb{R}^{d})$ satisfying a polynomial
growth condition, {where }$(\xi,\zeta)$ is {a pair of }$G$-distributed random
vectors and {where the sublinear function }$G:\mathbb{S}(d)\times
\mathbb{R}^{d}\mapsto \mathbb{R}$ is defined by{ }%
\[
G(p,A):=\mathbb{\hat{E}}[\left \langle p,Y_{1}\right \rangle +\frac{1}%
{2}\left \langle AX_{1},X_{1}\right \rangle ]{,\  \ }A\in \mathbb{S}%
(d),\  \ p\in \mathbb{R}^{d}.
\]

\end{theorem}

\begin{corollary}
The sum $\sum_{i=1}^{n}\frac{X_{i}}{\sqrt{n}}$ converges in law to
$\mathcal{N}(0,\hat{\Theta})$, where the subset $\hat{\Theta}\subset
\mathbb{S}_{+}(d)$ is defined in (\ref{Theta}) for $\hat{G}(A)=G(0,A)$,
$A\in \mathbb{S}(d)$. The sum $\sum_{i=1}^{n}\frac{Y_{i}}{n}$ converges in law
to $\mathcal{U}(\bar{\Theta})$, where the subset $\bar{\Theta}\subset
\mathbb{R}^{d}$ is defined in (\ref{thetabar}) for $\bar{G}(p)=G(p,0)$,
$p\in \mathbb{R}^{d}$. If we take in particular $\varphi(y)=d_{\bar{\Theta}%
}(y)=\inf \{|x-y|:x\in \bar{\Theta}\}$, then by (\ref{supOmega}) we have the
following generalized law of large number:%
\begin{equation}
\lim_{n\rightarrow \infty}\mathbb{\hat{E}}[d_{\bar{\Theta}}(\sum_{i=1}^{n}%
\frac{Y_{i}}{n})]=\sup_{\theta \in \bar{\Theta}}d_{\bar{\Theta}}(\theta)=0.
\label{LLN}%
\end{equation}

\end{corollary}

\begin{remark}
If $Y_{i}$ has no mean-uncertainty, or in other words, $\bar{\Theta}$ is a
singleton: $\bar{\Theta}=\{ \bar{\theta}\}$ then (\ref{LLN}) becomes
\[
\lim_{n\rightarrow \infty}\mathbb{\hat{E}}[|\sum_{i=1}^{n}\frac{Y_{i}}{n}%
-\bar{\theta}|]=0.
\]
To our knowledge, the law of large numbers with non-additive probability
measures have been investigated with a quite different framework and approach
from ours (see \cite{Marinacci}, \cite{Marinacci1}).
\end{remark}

To prove this theorem we first give

\begin{lemma}
\label{Lem-CLT}We \ assume the same condition as Theorem \ref{CLT}. We assume
furthermore that there exists $\beta>0$ such that, for each $A$, $\bar{A}%
\in \mathbb{S}(d)$ with $A\geq \bar{A}$, we have {\ }%
\begin{equation}
\mathbb{\hat{E}}[\left \langle AX_{1},X_{1}\right \rangle ]-\mathbb{\hat{E}%
}[\left \langle \bar{A}X_{1},X_{1}\right \rangle ]\geq \beta \text{tr}[A-\bar{A}].
\label{Ellip}%
\end{equation}
Then (\ref{e12}) holds.
\end{lemma}

\begin{proof}
We first prove (\ref{e12}) for $\varphi \in C_{b.Lip}(\mathbb{R}^{d})$. {For a
small but fixed $h>0$, let $V$ be the unique viscosity solution of}%
\begin{equation}
\partial_{t}V+G(DV,D^{2}V)=0,\ (t,x)\in \lbrack0,1+h]\times \mathbb{R}%
^{d}\text{,}\  \ V|_{t=1+h}=\varphi. \label{e14}%
\end{equation}
{Since }$(\xi,\zeta)$ is $G$-distributed{ we have }%
\begin{equation}
V(h,0)=\widetilde{\mathbb{E}}[\varphi(\xi+\zeta)],\  \ V(1+h,x)=\varphi(x)
\label{equ-h}%
\end{equation}
{ Since (\ref{e14}) is a uniformly parabolic PDE and $G$ is a convex function
thus, by the interior regularity of $V$ (see Krylov \cite{Krylov}, Example
6.1.8 and Theorem 6.2.3), we have
\[
\left \Vert V\right \Vert _{C^{1+\alpha/2,2+\alpha}([0,1]\times \mathbb{R}^{d}%
)}<\infty,\  \text{for some }\alpha \in(0,1).
\]
We set $\delta=\frac{1}{n}$ and $S_{0}=0$. Then \
\begin{align*}
&  V(1,\bar{S}_{n})-V(0,0)=\sum_{i=0}^{n-1}\{V((i+1)\delta,\bar{S}%
_{i+1})-V(i\delta,\bar{S}_{i})\} \\
&  =\sum_{i=0}^{n-1}{\large \{}[V((i+1)\delta,\bar{S}_{i+1})-V(i\delta,\bar
{S}_{i+1})]+[V(i\delta,\bar{S}_{i+1})-V(i\delta,\bar{S}_{i})]{\large \}}\\
&  =\sum_{i=0}^{n-1}\left \{  I_{\delta}^{i}+J_{\delta}^{i}\right \}
\end{align*}
with, by Taylor's expansion,%
\[
J_{\delta}^{i}=\partial_{t}V(i\delta,\bar{S}_{i})\delta+\frac{1}%
{2}\left \langle D^{2}V(i\delta,\bar{S}_{i})X_{i+1},X_{i+1}\right \rangle
\delta+\left \langle DV(i\delta,\bar{S}_{i}),X_{i+1}\sqrt{\delta}+Y_{i+1}%
\delta \right \rangle
\]%
\begin{align*}
&  I_{\delta}^{i}=\int_{0}^{1}[\partial_{t}V((i+\beta)\delta,\bar{S}%
_{i+1})-\partial_{t}V(i\delta,\bar{S}_{i+1})]d\beta \delta+[\partial
_{t}V(i\delta,\bar{S}_{i+1})-\partial_{t}V(i\delta,\bar{S}_{i})]\delta \\
&  +\frac{1}{2}\left \langle D^{2}V(i\delta,\bar{S}_{i})X_{i+1},Y_{i+1}%
\right \rangle \delta^{3/2}+\frac{1}{2}\left \langle D^{2}V(i\delta,\bar{S}%
_{i})Y_{i+1},Y_{i+1}\right \rangle \delta \\
&  +\int_{0}^{1}\int_{0}^{1}\left \langle \Theta_{\beta \gamma}^{i}(X_{i+1}%
\sqrt{\delta}+Y_{i+1}\delta),X_{i+1}\sqrt{\delta}+Y_{i+1}\delta \right \rangle
\gamma d\beta d\gamma
\end{align*}
with}%
\[
\Theta_{\beta \gamma}^{i}=D^{2}V(i\delta,\bar{S}_{i}+\gamma \beta(X_{i+1}%
\sqrt{\delta}+Y_{i+1}\delta)-D^{2}V(i\delta,\bar{S}_{i}).
\]
{Thus
\begin{equation}
\mathbb{\hat{E}}[\sum_{i=0}^{n-1}J_{\delta}^{i}]-\mathbb{\hat{E}}[-\sum
_{i=0}^{n-1}I_{\delta}^{i}]\leq \mathbb{\hat{E}}[V(1,\bar{S}_{n})]-V(0,0)\leq
\mathbb{\hat{E}}[\sum_{i=0}^{n-1}J_{\delta}^{i}]+\mathbb{\hat{E}}[\sum
_{i=0}^{n-1}I_{\delta}^{i}]. \label{e15}%
\end{equation}
We now prove that $\mathbb{\hat{E}}[\sum_{i=0}^{n-1}J_{\delta}^{i}]=0$. For
the 3rd term of $J_{\delta}^{i}$ we have:
\[
\mathbb{\hat{E}}[\left \langle DV(i\delta,\bar{S}_{i}),X_{i+1}\sqrt{\delta
}\right \rangle ]=\mathbb{\hat{E}}[-\left \langle DV(i\delta,\bar{S}%
_{i}),X_{i+1}\sqrt{\delta}\right \rangle ]=0.
\]
For the second term, we have, from the definition of the function }$G$,{
\[
\mathbb{\hat{E}}[J_{\delta}^{i}]=\mathbb{\hat{E}}[\partial_{t}V(i\delta
,\bar{S}_{i})+G(DV(i\delta,\bar{S}_{i}),D^{2}V(i\delta,\bar{S}_{i}))]\delta.
\]
We then combine the above two equalities with $\partial_{t}V+G(DV,D^{2}V)=0$
as well as the independence of $(X_{i+1},Y_{i+1})$ to $\{(X_{1},Y_{1}%
),\cdots,(X_{i},Y_{i})\}$, it follows that
\[
\mathbb{\hat{E}}[\sum_{i=0}^{n-1}J_{\delta}^{i}]=\mathbb{\hat{E}}[\sum
_{i=0}^{n-2}J_{\delta}^{i}]=\cdots=0.
\]
Thus (\ref{e15}) can be rewritten as%
\[
-\mathbb{\hat{E}}[-\sum_{i=0}^{n-1}I_{\delta}^{i}]\leq \mathbb{\hat{E}%
}[V(1,\bar{S}_{n})]-V(0,0)\leq \mathbb{\hat{E}}[\sum_{i=0}^{n-1}I_{\delta}%
^{i}].
\]
But since both $\partial_{t}V$ and $D^{2}V$ are uniformly $\alpha$-h\"{o}lder
continuous in $x$ and\ $\frac{\alpha}{2}$-h\"{o}lder continuous in $t$ on
$[0,1]\times \mathbb{R}$, we then have }%
\[
{|I_{\delta}^{i}|\leq C\delta^{1+\alpha/2}[1+|X_{i+1}|^{2+\alpha}%
+|Y_{1}|^{2+\alpha}].}%
\]
{ It follows that
\[
\mathbb{\hat{E}}[|I_{\delta}^{i}|]\leq C\delta^{1+\alpha/2}(1+\mathbb{\hat{E}%
}[|X_{1}|^{2+\alpha}+|Y_{1}|^{2+\alpha}]).
\]
Thus%
\begin{align*}
-C(\frac{1}{n})^{\alpha/2}(1+\mathbb{\hat{E}}[|X_{1}|^{2+\alpha}%
+|Y_{1}|^{2+\alpha}])  &  \leq \mathbb{\hat{E}}[V(1,\bar{S}_{n})]-V(0,0)\\
&  \leq C(\frac{1}{n})^{\alpha/2}(1+\mathbb{\hat{E}}[|X_{1}|^{2+\alpha}%
+|Y_{1}|^{2+\alpha}]).
\end{align*}
As $n\rightarrow \infty$ we have
\begin{equation}
\lim_{n\rightarrow \infty}\mathbb{\hat{E}}[V(1,\bar{S}_{n})]=V(0,0).
\label{equ-0}%
\end{equation}
On the other hand, for each $t,t^{\prime}\in \lbrack0,1+h]$ and $x\in
\mathbb{R}^{d}$, we have
\begin{align*}
|V(t,x)-V(t^{\prime},x)|  &  \leq C(\sqrt{|t-t^{\prime}|}.+|t-t^{\prime}|).\\
&  ,
\end{align*}
Thus $|V(0,0)-V(h,0)|\leq C$}$(\sqrt{h}+h)${ and, by (\ref{equ-0}),
\[
|\mathbb{\hat{E}}[V(1,\bar{S}_{n})]-\mathbb{\hat{E}}[\varphi(\bar{S}%
_{n})]|=|\mathbb{\hat{E}}[V(1,\bar{S}_{n})]-\mathbb{\hat{E}}[V(1+h,\bar{S}%
_{n})]|\leq C(\sqrt{h}+h).
\]
It follows form (\ref{equ-h}) and (\ref{equ-0}) that
\[
\limsup_{n\rightarrow \infty}|\mathbb{\hat{E}}[\varphi(\bar{S}_{n}%
)]-\widetilde{\mathbb{E}}[\varphi(\xi+\zeta)]|\leq2C(\sqrt{h}+h).
\]
Since $h$ can be arbitrarily small we thus have%
\[
\lim_{n\rightarrow \infty}\mathbb{\hat{E}}[\varphi(\bar{S}_{n})]=\widetilde
{\mathbb{E}}[\varphi(\xi)].
\]
}
\end{proof}

We now give

\  \

\noindent \textbf{Proof of Theorem \ref{CLT}.} For the case when the uniform
Elliptic condition \ref{Ellip} does not hold, we first follow an idea of Song
\cite{SongY} to\textbf{ }introduce a perturbation to prove the above
convergence for $\varphi \in C_{b.Lip}(\mathbb{R}^{d})$. According to
Definition \ref{Def-2-10} and Proposition \ref{Prop-2-10} we can construct a
sublinear expectation space $(\bar{\Omega},\mathcal{\bar{H}},\mathbb{\bar{E}%
)}$ and a sequence of three random vectors $\{(\bar{X}_{i},\bar{Y}_{i}%
,\bar{\eta}_{i})\}_{i=1}^{\infty}$ such that, for each $n=1,2,\cdots$,
$\{(\bar{X}_{i},\bar{Y}_{i})\}_{i=1}^{n}\overset{d}{=}\{(X_{i},Y_{i}%
)\}_{i=1}^{n}$ and $(\bar{X}_{n+1},\bar{Y}_{n+1},\bar{\eta}_{n+1})$ is
independent to $\{(\bar{X}_{i},\bar{Y}_{i},\bar{\eta}_{i})\}_{i=1}^{n}$ and,
moreover,
\[
\mathbb{\bar{E}}[\psi(X_{i},Y_{i},\eta_{i})]=\frac{1}{\sqrt{2\pi d}}%
\int_{\mathbb{R}^{d}}\mathbb{\hat{E}}[\psi(X_{i},Y_{i},x)]e^{-\frac{|x|^{2}%
}{2}}dx,\  \  \forall \psi \in C_{l.Lip}(\mathbb{R}^{3\times d}).
\]
We then use the following perturbation $X_{i}^{\varepsilon}=X_{i}%
+\varepsilon \eta_{i}$ for a fixed $\varepsilon>0$. It is seen that the
sequence $\{(\bar{X}_{i}^{\varepsilon},\bar{Y}_{i})\}_{i=1}^{\infty}$
satisfies all conditions in the above CLT, in particular%
\[
G_{\varepsilon}(p,A):=\mathbb{\bar{E}}[\frac{1}{2}\left \langle A\bar{X}%
_{1}^{\varepsilon},\bar{X}_{1}^{\varepsilon}\right \rangle +\left \langle
p,\bar{Y}_{1}\right \rangle ]=G(p,A)+\frac{\varepsilon^{2}}{2}\text{tr}[A].
\]
Thus it is strictly elliptic. We then can apply Lemma \ref{Lem-CLT} to%
\[
\bar{S}_{n}^{\varepsilon}:=\sum_{i=1}^{n}(\frac{X_{i}^{\varepsilon}}{\sqrt{n}%
}+\frac{Y_{i}}{n})=\bar{S}_{n}+\varepsilon J_{n},\  \ J_{n}=\sum_{i=1}^{n}%
\frac{\eta_{i}}{\sqrt{n}}%
\]
and obtain%
\[
\lim_{n\rightarrow \infty}\mathbb{\hat{E}}[\varphi(\bar{S}_{n}^{\varepsilon
})]=\widetilde{\mathbb{E}}[\varphi(\xi+\zeta+\varepsilon \eta)].
\]
where $(\xi,\zeta)$ is $G$-distributes and
\[
\widetilde{\mathbb{E}}[\psi(\xi+\zeta,\eta)]=\frac{1}{\sqrt{2\pi d}}%
\int_{\mathbb{R}^{d}}\widetilde{\mathbb{E}}[\psi(\xi+\zeta,x)]e^{\frac
{-|x|^{2}}{2}}dx,\  \  \psi \in C_{l.Lip}(\mathbb{R}^{2d}).
\]
Thus $(\xi+\varepsilon \eta,\zeta)$ is $G_{\varepsilon}$-distributed. But we
have%
\begin{align*}
|\mathbb{\hat{E}}[\varphi(\bar{S}_{n})]-\mathbb{\hat{E}}[\varphi(\bar{S}%
_{n}^{\varepsilon})]| &  =|\mathbb{\hat{E}}[\varphi(\bar{S}_{n})]-\mathbb{\hat
{E}}[\varphi(\bar{S}_{n}+\varepsilon J_{n})]|\\
&  \leq \varepsilon C\mathbb{\hat{E}}[|J_{n}|]\leq C\varepsilon
\end{align*}
and similarly $|\widetilde{\mathbb{E}}[\varphi(\xi)]-\widetilde{\mathbb{E}%
}[\varphi(\xi+\varepsilon \eta)]|\leq C\varepsilon$. Sine $\varepsilon$ can be
arbitrarily small, it follows that
\[
\lim_{n\rightarrow \infty}\mathbb{\hat{E}}[\varphi(\bar{S}_{n})]=\widetilde
{\mathbb{E}}[\varphi(\xi+\zeta)],\  \  \forall \varphi \in C_{b.Lip}%
(\mathbb{R}^{d}).
\]

On the other hand, it is easy to check that $\sup_{n}\mathbb{\hat{E}[}|\bar
{S}_{n}|]+\widetilde{\mathbb{E}}[|\xi+\zeta|]<\infty$. {We then can apply the
following lemma to prove that the above convergence holds for the case where
$\varphi$ in $C(\mathbb{R}^{d})$ with a polynomial growth condition. The proof
is complete.}

\begin{lemma}
Let $(\hat{\Omega},\mathcal{\hat{H}},\mathbb{\hat{E}})$ and $(\widetilde
{\Omega},\mathcal{\widetilde{H}},\mathbb{\widetilde{E}})$ be two sublinear
expectation space and let $\zeta \in \mathcal{\hat{H}}$ and $\zeta_{n}%
\in \mathcal{\widetilde{H}}$, $n=1,2,\cdots$, be given. We assume that, for a
given $p\geq1$ we have $\sup_{n}\mathbb{\hat{E}}[|Y_{n}|^{p}]+\widetilde
{\mathbb{E}}[|Y|^{p}]\leq C$. If the convergence $\lim_{n\rightarrow \infty
}\mathbb{\hat{E}}[\varphi(Y_{n})]=\widetilde{\mathbb{E}}[\varphi(Y)]$ holds
for each $\varphi \in C_{b.Lip}(\mathbb{R}^{d})$, then it also holds for all
function $\varphi \in C(\mathbb{R}^{d})$ with growth condition $|\varphi
(x)|\leq C(1+|x|^{p-1})$.
\end{lemma}

\begin{proof}
We first prove that the above convergence holds for $\varphi \in C_{b}%
(\mathbb{R}^{d})$ with a compact support. In this case, for each
$\varepsilon>0$, we can find a $\bar{\varphi}\in C_{b.Lip}(\mathbb{R}^{d})$
such that $\sup_{x\in \mathbb{R}^{d}}|\varphi(x)-\bar{\varphi}(x)|\leq
\frac{\varepsilon}{2}$. {We have } { } {
\begin{align*}
&  |\mathbb{\hat{E}}[\varphi(Y_{n})]-\widetilde{\mathbb{E}}[\varphi
(Y)]|\leq|\mathbb{\hat{E}}[\varphi(Y_{n})]-\mathbb{\hat{E}}[\bar{\varphi
}(Y_{n})]|+|\widetilde{\mathbb{E}}[\varphi(Y)]-\widetilde{\mathbb{E}}%
[\bar{\varphi}(Y)]|\\
&  +|\mathbb{\hat{E}}[\bar{\varphi}(Y_{n})]-\widetilde{\mathbb{E}}%
[\bar{\varphi}(Y_{n})]|\leq \varepsilon+|\mathbb{\hat{E}}[\bar{\varphi}%
(Y_{n})]-\widetilde{\mathbb{E}}[\bar{\varphi}(Y)]|.
\end{align*}
Thus }$\limsup_{n\rightarrow \infty}|\mathbb{\hat{E}}[\varphi(Y_{n}%
)]-\widetilde{\mathbb{E}}[\varphi(Y)]|\leq \varepsilon$. The convergence must
hold since $\varepsilon$ can be arbitrarily small.

Now let $\varphi$ be an arbitrary $C_{b}(\mathbb{R}^{n})$-function. For each
$N>0$ we can find $\varphi_{1},\varphi_{2}\in C_{b}(\mathbb{R}^{d})$ such that
$\varphi=\varphi_{1}+\varphi_{2}$ where $\varphi_{1}$ has a compact support
and $\varphi_{2}(x)=0$ for $|x|\leq N$, and $|\varphi_{2}(x)|\leq|\varphi(x)|$
for all $x$. It is clear that
\[
|\varphi_{2}(x)|\leq \frac{\bar{C}(1+|x|^{p})}{N},\  \  \forall
x,\  \  \text{where}\  \  \bar{C}=\sup_{x\in \mathbb{R}^{d}}|\varphi(x)|.
\]
Thus
\begin{align*}
|\mathbb{\hat{E}}[\varphi(Y_{n})]-\widetilde{\mathbb{E}}[\varphi(Y)]|  &
=|\mathbb{\hat{E}}[\varphi_{1}(Y_{n})+\varphi_{2}(Y_{n})]-\widetilde
{\mathbb{E}}[\varphi_{1}(Y)+\varphi_{2}(Y)]|\\
&  \leq|\mathbb{\hat{E}}[\varphi_{1}(Y_{n})]-\widetilde{\mathbb{E}}%
[\varphi_{1}(Y)]|+\mathbb{\hat{E}}[|\varphi_{2}(Y_{n})|]+\widetilde
{\mathbb{E}}[|\varphi_{2}(Y)|]\\
&  \leq|\mathbb{\hat{E}}[\varphi_{1}(Y_{n})]-\widetilde{\mathbb{E}}%
[\varphi_{1}(Y)]|+\frac{\bar{C}}{N}(\mathbb{\hat{E}}[|Y_{n}|+\widetilde
{\mathbb{E}}[|Y|])\\
&  \leq|\mathbb{\hat{E}}[\varphi_{1}(Y_{n})]-\widetilde{\mathbb{E}}%
[\varphi_{1}(Y)]|+\frac{\bar{C}C}{N}.
\end{align*}
We thus have $\limsup_{n\rightarrow \infty}|\mathbb{\hat{E}}[\varphi
(Y_{n})]-\widetilde{\mathbb{E}}[\varphi(Y)]|\leq \frac{\bar{C}C}{N}$. Since $N$
can be arbitrarily large thus $\mathbb{\hat{E}}[\varphi(Y_{n})]$ must converge
to $\widetilde{\mathbb{E}}[\varphi(Y)]$.
\end{proof}

\section{Appendix: some basic results of viscosity solutions}

We will use the following well-known result in viscosity solution theory (see
Theorem 8.3 of Crandall Ishii and Lions \cite{CIL}).

\begin{theorem}
\label{Thm-8.3} Let $u_{i}\in$USC$((0,T)\times Q_{i})$ for $i=1,\cdots,k$
where $Q_{i}$ is a locally compact subset of $\mathbb{R}^{N_{i}}$. Let
$\varphi$ be defined on an open neighborhood of $(0,T)\times Q_{1}\times
\cdots \times Q_{k}$ and such that $(t,x_{1},\cdots,x_{k})$ is once
continuously differentiable in $t$ and twice continuously differentiable in
$(x_{1},\cdots,x_{k})\in Q_{1}\times \cdots \times Q_{k}$. Suppose that $\hat
{t}\in(0,T)$, $\hat{x}_{i}\in Q_{i}$ for $i=1,\cdots,k$ and
\begin{align*}
w(t,x_{1},\cdots,x_{k})  &  :=u_{1}(t,x_{1})+\cdots+u_{k}(t,x_{k}%
)-\varphi(t,x_{1},\cdots,x_{k})\\
&  \leq w(\hat{t},\hat{x}_{1},\cdots,\hat{x}_{k})
\end{align*}
for $t\in(0,T)$ and $x_{i}\in Q_{i}$. Assume, moreover that there is an $r>0$
such that for every $M>0$ there is a $C$ such that for $i=1,\cdots,k$,%
\begin{equation}%
\begin{array}
[c]{ll}
& b_{i}\leq C\text{ whenever \ }(b_{i},q_{i},X_{i})\in \mathcal{P}^{2,+}%
u_{i}(t,x_{i}),\\
& |x_{i}-\hat{x}_{i}|+|t-\hat{t}|\leq r\text{ and }|u_{i}(t,x_{i}%
)|+|q_{i}|+\left \Vert X_{i}\right \Vert \leq M.
\end{array}
\label{eq8.5}%
\end{equation}
Then for each $\varepsilon>0$, there are $X_{i}\in \mathbb{S}(N_{i})$ such
that\newline \textsl{(i)} $(b_{i},D_{x_{i}}\varphi(\hat{t},\hat{x}_{1}%
,\cdots,\hat{x}_{k}),X_{i})\in \overline{\mathcal{P}}^{2,+}u_{i}(\hat{t}%
,\hat{x}_{i}),\  \ i=1,\cdots,k;$\newline \textsl{(ii)}
\[
-(\frac{1}{\varepsilon}+\left \Vert A\right \Vert )\leq \left[
\begin{array}
[c]{ccc}%
X_{1} & \cdots & 0\\
\vdots & \ddots & \vdots \\
0 & \cdots & X_{k}%
\end{array}
\right]  \leq A+\varepsilon A^{2}%
\]
\textsl{(iii)} $b_{1}+\cdots+b_{k}=\partial_{t}\varphi(\hat{t},\hat{x}%
_{1},\cdots,\hat{x}_{k})$\newline where $A=D^{2}\varphi(\hat{x})\in
\mathbb{S}(N_{1}+\cdots+N_{k})$.
\end{theorem}

Observe that the above conditions (\ref{eq8.5}) will be guaranteed by having
$u_{i}$ be subsolutions of parabolic equations given in the following two
theorems, which is an improved version of the one in the Appendix of
\cite{Peng2007b}.

\begin{theorem}
\textbf{\label{Thm-dom1}}(Domination Theorem) $u_{i}\in$\textrm{USC}%
$([0,T]\times \mathbb{R}^{N})$ be subsolutions of
\begin{equation}
\partial_{t}u-G_{i}(t,x,u,Du,D^{2}u)=0,\  \  \  \ i=1,\cdots,k, \label{visPDE0}%
\end{equation}
on $(0,T)\times \mathbb{R}^{N}$ such that, for given constants $\beta_{i}>0$,
$i=1,\cdots,k$, $\left(  \sum_{i=1}^{k}u_{i}(t,x)\right)  ^{+}\rightarrow0$,
uniformly as $|x|\rightarrow \infty$. We assume that \newline(i) The functions
\[
G_{i}:[0,T]\times \mathbb{R}^{N}\times \mathbb{R}^{N}\times \mathbb{S(}%
N)\mapsto \mathbb{R},\  \  \ i=1,\cdots,k,
\]
are continuous in the following sense: for each $t\in \lbrack0,T)$, $v_{1}$,
$v_{2}\in \mathbb{R}$, $x$, $y$, $p$, $q\in \mathbb{R}^{N}$ and $Y\in
\mathbb{S}(N)$,
\begin{align*}
&  [G_{i}(t,x,v,p,X)-G_{i}(t,y,v,p,X)]^{-}\\
&  \leq \bar{\omega}(1+(T-t)^{-1}+|x|+|y|+|v|)\omega(|x-y|+|p|\cdot|x-y|)
\end{align*}
where $\omega$, $\bar{\omega}:\mathbb{R}^{+}\mapsto \mathbb{R}^{+}$ are given
continuous functions with $\omega(0)=0$. $_{{}}$\newline(ii) Given constants
$\beta_{i}>0$, $i=1,\cdots,k$, the following domination condition holds for
$G_{i}$:
\begin{equation}
\sum_{i=1}^{k}\beta_{i}G_{i}(t,x,v_{i},p_{i},X_{i})\leq0,\  \  \text{ }
\label{domi-cond}%
\end{equation}
for each $(t,x)\in(0,T)\times \mathbb{R}^{N}$ and $(v_{i},p_{i},X_{i}%
)\in \mathbb{R}\times \mathbb{R}^{N}\times \mathbb{S}(N)$ such that\ $\sum
_{i=1}^{k}\beta_{i}v_{i}\geq0,\  \sum_{i=1}^{k}\beta_{i}p_{i}=0,\  \sum
_{i=1}^{k}\beta_{i}X_{i}\leq0.$

Then a similar domination also holds for the solutions: If the sum of initial
values $\sum_{i=1}^{k}\beta_{i}u_{i}(0,\cdot)$ is a non-positive and
continuous function on $\mathbb{R}^{N}$, then $\sum_{i=1}^{k}\beta_{i}%
u_{i}(t,\cdot)\leq0$, for all $t>0$.
\end{theorem}

\begin{proof}
\medskip \medskip We first observe that for $\bar{\delta}>0$ and for each
$1\leq i\leq k$, the functions defined by $\tilde{u}_{i}:=u_{i}-\bar{\delta
}/(T-t)$ is a subsolution of:%
\[
\partial_{t}\tilde{u}_{i}-\tilde{G}_{i}(t,x,\tilde{u}_{i}+\bar{\delta
}/(T-t),D\tilde{u}_{i},D^{2}\tilde{u}_{i})\leq-\frac{\bar{\delta}}{(T-t)^{2}}%
\]
where $\tilde{G}_{i}(t,x,v,p,X):=G_{i}(t,x,v+\bar{\delta}/(T-t),p,X)$. It is
easy to check that the functions $\tilde{G}_{i}$ satisfy the same conditions
as $G_{i}$. Since $\sum_{i=1}^{k}\beta_{i}u_{i}\leq0$ follows from $\sum
_{i=2}^{k}\beta_{i}\tilde{u}_{i}\leq0$ in the limit $\bar{\delta}\downarrow0$,
it suffices to prove the theorem under the additional assumptions%
\begin{equation}%
\begin{array}
[c]{c}%
\partial_{t}u_{i}-G_{i}(t,x,u_{i},Du_{i},D^{2}u_{i})\leq-c,\  \  \ c:=\bar
{\delta}/T^{2},\\
\  \  \  \  \  \  \  \  \  \  \  \  \  \text{and }\lim_{t\rightarrow T}u_{i}(t,x)=-\infty
,\  \  \  \  \text{uniformly in }[0,T)\times \mathbb{R}^{N}.
\end{array}
\label{ineq-c}%
\end{equation}
\newline To prove the theorem, we assume to the contrary that
\[
\sup_{(s,x)\in \lbrack0,T)\times \mathbb{R}^{N}}\sum_{i=1}^{k}\beta_{i}%
u_{i}(t,x)=m_{0}>0
\]
We will apply Theorem \ref{Thm-8.3} for $x=(x_{1},\cdots,x_{k})\in
\mathbb{R}^{k\times N}$ and%
\[
w(t,x):=\sum_{i=1}^{k}\beta_{i}u_{i}(t,x_{i}),\  \  \varphi(x)=\varphi_{\alpha
}(x):=\frac{\alpha}{2}\sum_{i=1}^{k-1}|x_{i+1}-x_{i}|^{2}.
\]
For each large $\alpha>0$ the maximum of $w-\varphi_{\alpha}$ achieved at some
$(t^{\alpha},x^{\alpha})$ inside a compact subset of $[0,T)\times
\mathbb{R}^{k\times N}$. Indeed, since
\[
M_{\alpha}=\sum_{i=1}^{k}\beta_{i}u_{i}(t^{\alpha},x_{i}^{\alpha}%
)-\varphi_{\alpha}(x^{\alpha})\geq m_{0},
\]
thus $t^{\alpha}$ must be inside an interval $[0,T_{0}]$, $T_{0}<T$ and
$x^{\alpha}$ must be inside a compact set $\{x\in \times \mathbb{R}^{k\times
N}:\sup_{t\in \lbrack0,T_{0}]}w(t,x)\geq \frac{m_{0}}{2}\}$. We can check that
(see \cite{CIL} Lemma 3.1)
\begin{equation}
\left \{
\begin{array}
[c]{l}%
\text{(i) }\lim_{\alpha \rightarrow \infty}\varphi_{\alpha}(x^{\alpha
})=0\text{.}\\
\text{(ii)\ }\lim_{\alpha \rightarrow \infty}M_{\alpha}=\lim_{\alpha
\rightarrow \infty}\beta_{1}u_{1}(t^{\alpha},x_{1}^{\alpha})+\cdots+\beta
_{k}u_{k}(t^{\alpha},x_{k}^{\alpha}))\\
\; \; \; \; \; \; \; \;=\sup_{(t,x)\in \lbrack0,T)\times \mathbb{R}^{N}}[\beta_{1}%
u_{1}(t,x)+\cdots+\beta_{k}u_{k}(t,x)]\\
\  \  \  \  \  \  \  \  \  \  \  \ =[\beta_{1}u_{1}(\hat{t},\hat{x})+\cdots+\beta
_{k}u_{k}(\hat{t},\hat{x})]=m_{0}.
\end{array}
\right.  \label{limit}%
\end{equation}
where $(\hat{t},\hat{x})$ is a limit point of $(t^{\alpha},x_{1}^{\alpha})$.
Since $u_{i}\in \mathrm{USC}$, for sufficiently large $\alpha$, we have
\[
\beta_{1}u_{1}(t^{\alpha},x_{1}^{\alpha})+\cdots+\beta_{k}u_{k}(t^{\alpha
},x_{k}^{\alpha})\geq \frac{m_{0}}{2}.
\]
If $\hat{t}=0$, we have $\limsup_{\alpha \rightarrow \infty}\sum_{i=1}^{k}%
\beta_{i}u_{i}(t^{\alpha},x_{i}^{\alpha})=\sum_{i=1}^{k}\beta_{i}u_{i}%
(0,\hat{x})\leq0$. We know that $\hat{t}>0$ and thus $t^{\alpha}$ must be
strictly positive for large $\alpha$. It follows from Theorem \ref{Thm-8.3}
that, for each $\varepsilon>0$ there exists $b_{i}^{\alpha}\in \mathbb{R}$,
$X_{i}\in \mathbb{S}(N)$ such that
\begin{equation}
(b_{i}^{\alpha},D_{x_{i}}\varphi(x^{\alpha}),X_{i})\in \bar{J}_{Q_{i}}%
^{2;+}(u_{i})(t^{\alpha},x_{i}^{\alpha}),\  \  \sum_{i=1}^{k}\beta_{i}%
b_{i}^{\alpha}=0,\; \text{for }i=1,\cdots,k,\label{b-p-eqn}%
\end{equation}
and such that
\begin{equation}
-(\frac{1}{\varepsilon}+\left \Vert A\right \Vert )I\leq \left(
\begin{array}
[c]{cccc}%
\beta_{1}X_{1} & \ldots & 0 & 0\\
\vdots & \ddots & \vdots & \vdots \\
0 & \ldots & \beta_{k-1}X_{k-1} & 0\\
0 & \ldots & 0 & \beta_{k}X_{k}%
\end{array}
\right)  \leq A+\varepsilon A^{2},\label{ine-matrix}%
\end{equation}
where $A=D^{2}\varphi_{\alpha}(x^{\alpha})\in \mathbb{S}(kN)$ is explicitly
given by%
\[
A=\alpha J_{kN},\; \text{where}\;J_{kN}=\left(
\begin{array}
[c]{cccc}%
I_{N} & -I_{N}\ 0\ldots & 0 & -I_{N}\\
\vdots & \ddots & \vdots & \vdots \\
0 & 0\ldots-I_{N} & I_{N} & -I_{N}\\
-I_{N} & 0\ldots0 & -I_{N} & I_{N}%
\end{array}
\right)  .
\]
The second inequality of (\ref{ine-matrix}) implies $\sum_{i=1}^{k}\beta
_{i}X_{i}\leq0$. Setting
\begin{align*}
p_{1}^{\alpha} &  =D_{x_{1}}\varphi_{\alpha}(x^{\alpha})=\beta_{1}^{-1}%
\alpha(2x_{1}^{\alpha}-x_{3}^{\alpha}-x_{2}^{\alpha}),\; \\
&  \  \  \vdots \\
p_{k}^{\alpha} &  =D_{x_{k}}\varphi_{\alpha}(x^{\alpha})=\beta_{k}^{-1}%
\alpha(2x_{k}^{\alpha}-x_{k-1}^{\alpha}-x_{1}^{\alpha}).
\end{align*}
Thus $\sum_{i=1}^{k}\beta_{i}p_{i}^{\alpha}=0$. This with (\ref{b-p-eqn}) and
(\ref{ineq-c}) it follows that
\[
b_{i}^{\alpha}-G_{i}(t^{\alpha},x_{i}^{\alpha},u_{i}(t^{\alpha},x_{i}^{\alpha
}),p_{i}^{\alpha},X_{i})\leq-c,\  \ i=1,\cdots,k.\
\]
By (\ref{limit})-(i) we also have $\lim_{\alpha \rightarrow \infty}%
|p_{i}^{\alpha}|\cdot|x_{i}^{\alpha}-x_{1}^{\alpha}|\rightarrow0$. This,
together with the the domination condition (\ref{domi-cond}) of $G_{i}$,
implies
\begin{align*}
-kc &  =-\sum_{i=1}^{k}\beta_{i}b_{i}^{\alpha}-kc\geq-\sum_{i=1}^{k}\beta
_{i}G_{i}(t^{\alpha},x_{i}^{\alpha},u_{i}(t^{\alpha},x_{i}^{\alpha}%
),p_{i}^{\alpha},X_{i})\\
&  \geq-\sum_{i=1}^{k}\beta_{i}G_{i}(t^{\alpha},x_{1}^{\alpha},u_{i}%
(t^{\alpha},x_{i}^{\alpha}),p_{i}^{\alpha},X_{i})\\
&  \  \  \  \  \ -\sum_{i=1}^{k}\beta_{i}[G_{i}(t^{\alpha},x_{i}^{\alpha}%
,u_{i}(t^{\alpha},x_{i}^{\alpha}),p_{i}^{\alpha},X_{i})-G_{i}(t^{\alpha}%
,x_{1}^{\alpha},u_{i}(t^{\alpha},x_{i}^{\alpha}),p_{i}^{\alpha},X_{i})]^{-}\\
&  \geq-\sum_{i=1}^{k}\beta_{i}\bar{\omega}(1+(T-T_{0})^{-1}+|x_{1}^{\alpha
}|+|x_{i}^{\alpha}|+|u_{i}(t^{\alpha},x_{i}^{\alpha})|)\omega(|x_{i}^{\alpha
}-x_{1}^{\alpha}|+|p_{i}^{\alpha}|\cdot|x_{i}^{\alpha}-x_{1}^{\alpha}|)\  \
\end{align*}
The right side tends to zero as $\alpha \rightarrow \infty$, which induces a
contradiction. The proof is complete.
\end{proof}

\begin{theorem}
\textbf{\label{Thm-dom}}(Domination Theorem) Let polynomial growth functions
$u_{i}\in$USC$([0,T]\times \mathbb{R}^{N})$ be subsolutions of
\begin{equation}
\partial_{t}u-G_{i}(u,Du,D^{2}u)=0,\  \  \  \ i=1,\cdots,k, \label{visPDE}%
\end{equation}
on $(0,T)\times \mathbb{R}^{N}$. We assume that $G_{i}:\mathbb{R}%
\times \mathbb{R}^{N}\times \mathbb{S(N})\mapsto \mathbb{R}$, $i=1,\cdots,k$, are
given continuous functions satisfying the following conditions: There exists a
positive constant $C$, such that\newline{\normalsize (i)} for all $\lambda
\geq0$
\[
\lambda G_{i}(v,p,Y)\leq CG_{i}(\lambda v,\lambda p,\lambda Y);\
\]
{\normalsize (ii)} Lipschitz condition:
\begin{align*}
|G_{i}(v_{1},p,X)-G_{i}(v_{2},q,Y)|  &  \leq C(|v_{1}-v_{2}|+|p-q|+\left \Vert
X-Y\right \Vert ),\  \\
\forall v_{1},v_{2}  &  \in \mathbb{R},\  \forall p,q\in \mathbb{R}^{N}\text{
\ and }X,Y\in \mathbb{S}(N),
\end{align*}
{\normalsize (iii)} domination condition for $G_{i}$: for fixed constants
$\beta_{i}>0$, $i=1,\cdots,k$,%
\begin{align*}
\sum_{i=1}^{k}\beta_{i}G_{i}(v_{i},p_{i},X_{i})  &  \leq0,\  \  \text{for all
}v_{i}\in \mathbb{R},\ p_{i}\in \mathbb{R}^{N},\text{ }X_{i}\in \mathbb{S}(N),\\
\  \text{such that }\sum_{i=1}^{k}\beta_{i}v_{i}  &  \geq0,\text{ }\sum
_{i=1}^{k}\beta_{i}p_{i}=0,\  \sum_{i=1}^{k}\beta_{i}X_{i}\leq0.
\end{align*}

Then the following domination holds: If $\sum_{i=1}^{k}\beta_{i}u_{i}%
(0,\cdot)$ is a non-positive continuous function, then we have
\[
\sum_{i=1}^{k}\beta_{i}u_{i}(t,x)\leq0,\  \  \  \forall(t,x)\in(0,T)\times
\mathbb{R}^{N}.
\]

\end{theorem}

\begin{proof}
We set $\xi(x):=(1+|x|^{2})^{l/2}$ and
\[
\tilde{u}_{i}(t,x):=u_{i}(t,x)e^{\lambda t}\xi^{-1}(x),\ i=1,\cdots,k,
\]
where $l$ is chosen large enough so that $\sum_{i=1}^{k}|\tilde{u}%
_{i}(t,x)|\rightarrow0$ uniformly. From condition (i) it is easy to check that
for each $i=1,\cdots$, $\tilde{u}_{i}$ is a subsolution of
\begin{equation}
\partial_{t}\tilde{u}_{i}-\tilde{G}_{i}(x,\tilde{u}_{i},D\tilde{u}_{i}%
,D^{2}\tilde{u}_{i})=0, \label{subF}%
\end{equation}
where
\begin{align*}
&  \tilde{G}_{i}(x,v,p,X):=-\lambda v\\
&  +e^{\lambda t}CG_{i}(v,p+v\eta(x),X+p\otimes \eta(x)+\eta(x)\otimes
p+v\kappa(x)).
\end{align*}
Here
\begin{align*}
\eta(x)  &  :=\xi^{-1}(x)D\xi(x)=k(1+|x|^{2})^{-1}x,\  \  \  \  \\
\kappa(x)  &  :=\xi^{-1}(x)D^{2}\xi(x)=k(1+|x|^{2})^{-1}I-k(k-2)(1+|x|^{2}%
)^{-2}x\otimes x.
\end{align*}
Since $\eta$ and $\kappa$ are uniformly bounded, one can choose a fixed but
large enough $\lambda>0$ such that $\tilde{G}_{i}(x,v,p,X)$ satisfies all
conditions of $G_{i}$, $i=1,\cdots,k$ in Theorem \ref{Thm-dom1}. The proof is
complete by directly applying this theorem.
\end{proof}

We have the following Corollaries which are basic in this paper:

\begin{corollary}
\label{Comparison}(Comparison Theorem) Let $F_{1},F_{2}:\mathbb{R}^{N}%
\times \mathbb{S}(N)\mapsto \mathbb{R}$ be given functions satisfying similar
conditions (i) and (ii) of Theorem \ref{Thm-dom}. We also assume that, for
each $p,q\in \mathbb{R}^{N}$ and $X$, $Y\in \mathbb{S}(N)$ such that $X\geq Y$,
we have
\[
F_{1}(p,X)\geq F_{2}(p,Y).
\]
Let $v_{i}\in$\textrm{LSC}$((0,T)\times \mathbb{R}^{N})$, $i=1,2$, be
respectively a viscosity supersolution of $\partial_{t}v-F_{i}(Dv,D^{2}v)=0$
such that $v_{1}(0,\cdot)-v_{2}(0,\cdot)$ is a non-negative continuous
function. Then we have $v_{1}(t,x)-v_{2}(t,x)\geq0$ for all $(t,x)\in
\lbrack0,\infty)\times \mathbb{R}^{N}$.
\end{corollary}

\begin{proof}
We set $\beta_{1}=\beta_{2}=1$, $G_{1}(p,X):=-F_{1}(-p,-X)$ and $G_{2}%
=F_{2}(p,X)$. It is observed that $u_{1}:=-v_{1}\in$\textrm{USC}%
$((0,T)\times \mathbb{R}^{N})$ is a viscosity subsolution of $\partial
_{t}u-G_{1}(Du,D^{2}u)=0$. For each $p_{1}$, $p_{2}\in \mathbb{R}^{N}$ and
$X_{1}$, $X_{2}\in \mathbb{S}(N)$ such that $p_{1}+p_{2}=0$ and $X_{1}%
+X_{2}\leq0$, we also have
\[
G_{1}(p_{1},X_{1})+G_{2}(p_{2},X_{2})=F_{2}(p_{2},X_{2})-F_{1}(p_{2}%
,-X_{1})\leq0
\]
We thus can apply Theorem \ref{Thm-dom} to get $-u_{1}+u_{2}\leq0$. The proof
is complete.
\end{proof}

\begin{corollary}
\label{Domination}(Domination Theorem) Let $F_{i}:\mathbb{R}^{N}%
\times \mathbb{S}(N)\mapsto \mathbb{R}$, $i=0,1$, given functions satisfying
similar conditions (i) and (ii) of Theorem \ref{Thm-dom}. Let $v_{i}\in
$\textrm{LSC}$((0,T)\times \mathbb{R}^{N})$ be viscosity supersolutions of
$\partial_{t}v-F_{i}(Dv,D^{2}v)=0$ respectively for $i=0,1$ and let $v_{2}\in
$\textrm{USC}$((0,T)\times \mathbb{R}^{N})$ be a viscosity subsolution of
$\partial_{t}v-F_{1}(Dv,D^{2}v)=0$. We assume that:
\begin{align*}
F_{1}(p,X)-F_{1}(q,Y)  &  \leq F_{0}(p-q,Z),\  \\
\  \forall p,q  &  \in \mathbb{R}^{N},\ X,Y,Z\text{ }\in \mathbb{S}%
(N)\  \text{such that }X-Y\leq Z.\
\end{align*}
Then the following domination holds: If $v_{0}(0,\cdot)+v_{1}(0,\cdot
)-v_{2}(0,\cdot)$ is a continuous and non-negative function then
$v_{0}(t,\cdot)+v_{1}(t,\cdot)-v_{2}(t,\cdot)\geq0$ for all $t>0$.
\end{corollary}

\begin{proof}
We denote
\[
G_{i}(p,X):=-F_{i}(-p,-X),\  \ i=0,1,\  \text{and }\  \ G_{2}(p,X):=F_{1}(p,X).
\]
Observe that $u_{i}=-v_{i}\in$\textrm{USC}$((0,T)\times \mathbb{R}^{N})$,
$i=0,1$, are viscosity subsolutions of $\partial_{t}u-G_{i}(Du,D^{2}u)=0$,
$i=0,1$. We thus have, for each $X_{0}+X_{1}+X_{2}\leq0$, $p_{0}+p_{1}%
+p_{2}=0$,
\begin{align*}
&  G_{0}(p_{0},X_{0})+G_{1}(p_{1},X_{1})+G_{2}(p_{2},X_{2})\\
&  =-F_{0}(-p_{0},-X_{0})-F_{1}(-p_{1},-X_{1})+F_{1}(p_{2},X_{2})\leq0.
\end{align*}
Theorem \ref{Thm-dom} can be applied, for the case $\beta_{i}=1$, to get $\sum
u_{i}\leq0$, or $v_{0}+v_{1}-v_{2}\geq0$.
\end{proof}

Another co-product of Theorem \ref{Thm-dom} is:

\begin{corollary}
(Concavity) Let $F:\mathbb{R}^{N}\times \mathbb{S}(N)\mapsto \mathbb{R}$ be a
given function satisfying similar conditions (i) and (ii) of Theorem
\ref{Thm-dom}. We assume that $F$ is monotone in $X$, i.e. $F(p,X)\geq F(p,Y)$
if $X\geq Y$, and that $F$ is concave: for each fixed $\alpha \in(0,1)$,
\[
\alpha F(p,X)+(1-\alpha)F(q,Y)\leq F(\alpha p+(1-\alpha)q,\alpha
X+(1-\alpha)Y),\  \  \forall p,q\in \mathbb{R}^{N},\ X,Y\in \mathbb{S}(N).\
\]
Let $v_{i}\in$\textrm{USC}$((0,T)\times \mathbb{R}^{N})$, $i=0,1$, be
respectively viscosity subsolutions of $\partial_{t}v-F(Dv,D^{2}v)=0$ and let
$v\in$\textrm{LSC}$((0,T)\times \mathbb{R}^{N})$ be viscosity supersolution of
$\partial_{t}v-F(Dv,D^{2}v)=0$ such that $\alpha v_{1}(0,\cdot)+(1-\alpha
)v_{2}(0,\cdot)-v(0,\cdot)$ is a non-positive continuous function. Then for
all $t\geq0$ $\alpha v_{1}(t,\cdot)+(1-\alpha)v_{2}(t,\cdot)-v(t,\cdot)\geq0$.
\end{corollary}

\begin{proof}
We set $\beta_{1}=\alpha$, $\beta_{2}=(1-\alpha)$, $\beta_{3}=1$ and denote
\[
G_{1}(p,X)=G_{2}(p,X):=F(p,X),\ G_{3}(p,X)=-F(-p,-X).
\]
Observe that $u_{i}=v_{i}\in$\textrm{USC}$((0,T)\times \mathbb{R}^{N})$,
$i=1,2$, are viscosity subsolutions of $\partial_{t}u-G_{i}(Du,D^{2}u)=0$,
$u_{3}=-v\in \mathrm{USC}$ is a viscosity subsolution of $\partial_{t}%
u-G_{3}(Du,D^{2}u)=0$. Since $F$ is concave, thus for each $p_{i}\in
\mathbb{R}^{N}$ and $X_{i}\in \mathbb{S}(N)$ such that $\beta_{1}X_{1}%
+\beta_{2}X_{2}+\beta_{3}X_{3}\leq0$, $\beta_{1}p_{1}+\beta_{2}p_{2}+\beta
_{3}p_{3}=0$, we have
\begin{align*}
\beta_{1}G_{1}(p_{1},X_{1})+\beta_{2}G_{2}(p_{2},X_{2})+\beta_{3}G_{3}%
(p_{3},X_{3})  &  \leq F(\beta_{1}p_{1}+\beta_{2}p_{2},\beta_{1}X_{1}%
+\beta_{2}X_{2})-F(-p_{3},-X_{3})\\
&  \leq F(-p_{3},\beta_{1}X_{1}+\beta_{2}X_{2})-F(-p_{3},-X_{3})\\
&  \leq0.
\end{align*}
Theorem \ref{Thm-dom} can be applied to prove that $\alpha v_{1}%
(t,\cdot)+(1-\alpha)v_{2}(t,\cdot)\leq v(t,\cdot)$. The proof is complete.
\end{proof}

\end{document}